\documentclass[12pt]{amsart}
\usepackage[left=1in,top=1in,right=1in,bottom=1in]{geometry}
\usepackage{mathrsfs}
\usepackage[english]{babel}
\usepackage{amsmath,amsthm,amsfonts,amssymb,epsfig, hyperref}
\usepackage{bbm,booktabs,float,mathtools,siunitx,tikz}
\usepackage[shortlabels]{enumitem}
\usepackage{breqn}
\usepackage{stmaryrd}
\usepackage{relsize}
﻿\usepackage{cancel}
\usepackage{ulem}

\def\Z{\ensuremath\mathbb{Z}}
\def\N{\ensuremath\mathbb{N}}
\def\R{\ensuremath\mathbb{R}}

\def\P{\ensuremath\mathbb{P}}

\def\E{\ensuremath\mathbb{E}}

\numberwithin{equation}{section}
\newtheorem{thm}{Theorem}

\newtheorem{lem}[thm]{Lemma}
\newtheorem{prop}[thm]{Proposition}

\newtheorem{df}[thm]{Definition}
\newtheorem*{rem}{Remark}

\numberwithin{thm}{section}

﻿

\newcommand{\ii}{{\rm i}}

\newcommand{\rd}{{\rm d}}

\newcommand{\re}{{\rm Re}}

\newcommand{\nocontentsline}[3]{}
\let\origcontentsline\addcontentsline
\newcommand\stoptoc{\let\addcontentsline\nocontentsline}
\newcommand\resumetoc{\let\addcontentsline\origcontentsline}
﻿

﻿\newtheorem{cor}[thm]{Corollary}

﻿
﻿

﻿
﻿
﻿
﻿
\begin{document}
﻿
﻿
\title[Lower bounds for large deviations of the zeta function on the critical line]{Lower bounds for the large deviations and moments of the Riemann zeta function on the critical line}
\author{Louis-Pierre Arguin}

\address{L.P. Arguin, Mathematical Institute, University of Oxford, UK, and Baruch College and Graduate Center, City University of New York, NY}
\email{louis-pierre.arguin@maths.ox.ac.uk}

\author{Nathan Creighton}
\address{N. Creighton, Mathematical Institute, University of Oxford, UK}
\email{nathan.creighton@maths.ox.ac.uk}
\maketitle

\begin{abstract}
Building on work in \cite{AB24} on the Riemann zeta function at height $T$ off the critical line, we prove an unconditional lower bound on the critical line for real large deviations of the order $V\sim\alpha\log\log T$ for any $\alpha>0.$  This gives another proof of the sharpest known unconditional lower bounds on the fractional moments of the Riemann zeta function, due to \cite{HSlower}. The lower bound on large deviations is of the same order of magnitude  as the upper bound proved in \cite{AB23}, for the range $0<\alpha<2.$
\end{abstract}

\tableofcontents

\section{Introduction}
Ever since Selberg's Central Limit Theorem \cite{Sel} for $\log\zeta\left(\frac{1}{2}+\ii t\right)$, the distribution of values of the Riemann zeta function on the critical line $\re \ s=1/2$ has been a key topic in analytic number theory.
It is of particular interest to understand the extent to which  the Gaussian behaviour persists beyond the standard deviation order, $\sqrt{\log\log T}$.
Theorem 2 in \cite{Radz} states that the measures of the level sets obey the same asymptotic for deviations of $\log|\zeta(\frac{1}{2}+\ii t)|$ up to $\left(\log\log T\right)^{\frac{3}{5}}$.
However, it is the range of the distribution at the order of the variance which determines the growth of the moments of the Riemann zeta function.
This provides the motivation for studying the measure of level sets at this range.
\subsection{Main results}
We prove a lower bound for the measure of level sets of the Riemann zeta function on the critical line at the order of magnitude $\sqrt{\log\log T}.$
\begin{thm}\label{LBLD}
Let $\alpha>0$ and $V\sim \alpha \log\log T.$ Then for $T$ sufficiently large, we have
\begin{equation}
\frac{1}{T}\left|\left\{t\in [0,T]: \log\left|\zeta\left(\frac{1}{2}+\ii t\right)\right|>V\right\}\right|\ge K_{\alpha}\int^\infty_V\frac{e^{-\frac{y^2}{\log\log T}}}{\sqrt{\log\log T}}\rd y,
\end{equation}
with the constant $K_\alpha$ satisfying
\begin{equation}\label{ca}
K_\alpha\gg\left(C\alpha^2\log\alpha\right)^{-\alpha^2},
\end{equation}
 as $\alpha\to\infty,$ for some absolute constant $C$.

\end{thm}
A matching upper bound up to a constant was proved unconditionally in \cite{AB23} in the range $0<\alpha<2$.
By using the method of decomposing moments into integrals over measures of level sets as pioneered by Soundararajan \cite{Sound}, we recover the lower bound for all the  fractional moments proved in \cite{HSlower} and \cite{RSlower}, with the same constant. 
\begin{cor}
\label{cor: moments}
Let $\alpha> 0$ be any positive real number and $T$ be sufficiently large.
Then for $K_\alpha$ as above,  \begin{equation}
\int^T_0|\zeta(1/2+\ii t)|^{2\alpha}\rd t \gg K_\alpha T\left(\log T\right)^{\alpha^2}.
\end{equation}
\end{cor}
In \cite{Sound}, Soundararajan mentions: \textit{Our proof suggests that the dominant contribution to the $2k^{\text{th}}$ moment comes from $t$ such that $\zeta(1/2+\ii t)$ has size $(\log T)^k,$ and this set has measure about $T/(\log T)^{k^2}$.}
Theorem \ref{LBLD} confirms the suggestion that the values of this range do occur with high enough measure to produce the moments.
More work is required to reverse the decomposition of the moments and attain results concerning the measure of level sets, even when the corresponding moments bounds are known. From the lower bound on the fractional moments, Soundararajan requires the Riemann Hypothesis to attain an upper bound as well as the following lower bound large deviation result in \cite[Corollary B]{Sound}:
\begin{equation}
\text{meas}\left\{t\in [0,T]: \left|\zeta\left(\frac{1}{2}+\ii t\right)\right|\ge \left(\log T\right)^k\right\}\ge T(\log T)^{-k^2+o(1)}.
\end{equation}

The main purpose of Theorem \ref{LBLD} is to remove the dependency on the Riemann Hypothesis and sharpen this to attain the unconditional Gaussian decay, but we go further and get constants necessary to retrieve the moments bound in Corollary \ref{cor: moments}.

Note that in \cite{AB24}, the authors produced a lower bound for large deviations shifted off the critical line to real part $\sigma_0=\frac{1}{2}+\frac{\delta(\alpha)}{\log T},$ for some sufficiently large parameter $\delta(\alpha).$ By a convexity argument, this was enough to recover the moment bounds on the critical line, albeit with a worse constant $\gg (C\alpha^{36}\log \alpha)^{-\alpha^2}$.
In contrast, Theorem \ref{LBLD} restores the results to large deviations on the critical line  with the same constant as the sharpest known results for the corresponding lower bounds for the moments.
Conjecture 2 in  \cite{Radz} predicts that, in fact, the constant for the large deviations of the range in Theorem \ref{LBLD} matches that from the conjectures of Keating and Snaith \cite{KS} for the moments.
More precisely, the constant should be the product of an arithmetic factor $a_\alpha$ and a geometric factor $g_\alpha$. As $\alpha\to\infty$, these factors behave like (\cite{CG},\cite{HR})  \begin{equation}\label{eqn: factors}
a_\alpha=(\log\alpha)^{-\alpha^2}\exp(O(\alpha^2)),\qquad g_\alpha=\alpha^{-\alpha^2}\exp(O(\alpha^2)).
\end{equation}
The constant $K_\alpha$ is approximately a square worse than this conjectured value.

\subsection{Results in the $q$-aspect}
The method of proof generalises to many other families of $L$-functions where two distinct twisted moments can be bounded, with a lower bound for the lower order moment and an upper bound for the higher. For example, one can use it to give lower bounds on large deviations of central values of primitive Dirichlet $L$-functions. For this, in place of the twisted second and fourth moments of the Riemann zeta function, we deploy the twisted second \cite{Sarnak},\cite{bui}  and fourth \cite{Zach19} moments of central values of  $L$-functions with even, primitive characters with the same large modulus. This yields the following theorem.
\begin{thm}\label{thm:Q}
Let $\alpha>0,$ $q$ be a large prime, and $W\sim \alpha \log\log q.$ Then
\begin{equation}\label{eqn:qLD}\frac{1}{\varphi^+(q)}\left|\left\{\chi \text{ even primitive modulo }q: \log\left|L\left(\frac{1}{2},\chi\right)\right|\ge W\right\}\right|\ge K_\alpha \int^\infty_W\frac{e^{-\frac{y^2}{\log\log q}}}{\sqrt{\log\log q}}\rd y,\end{equation}
where the constant $K_\alpha$ satisfies the bounds in Equation \eqref{ca}.
\end{thm}
This matches the upper bound proven in \cite[Theorem 1.1]{AC25} when $0<\alpha<1.$
In \cite[Section 7]{Sarnak}, it is shown that, when ranging over the even primitive characters with a large modulus $q,$ a positive proportion of central $L$-values are $\gg (\log q)^{-\frac{1}{2}}$ (and hence don't vanish). By using the Cauchy-Schwarz inequality to compare the first and second mollified moments of the central values, they conclude these results couldn't simultaneously hold without a positive  proportion of large central values. This theme runs throughout the proof of Theorem \ref{LBLD}, where we use the Cauchy-Schwarz inequality in the form of the Paley-Zygmund inequality (see \cite[Section 2.3.1]{Roch24}), which states that for a non-negative random variable $Z$ with finite second moment and a parameter $0<\delta<1,$ one has
\begin{equation}\label{eqn:Paley-Zygmund}
\P\left(Z\ge\delta \E\left[Z\right]\right)\ge (1-\delta)^2\frac{\E[Z]^2}{\E[Z^2]}.
\end{equation}

Theorem \ref{thm:Q} is consistent with  the Central Limit Theorem for central $L$-values with a large fixed modulus, see for example \cite[Theorem 4.2.1]{Dasthesis}. As in Corollary \ref{cor: moments}, this yields lower bounds on all the positive fractional moments of the central Dirichlet $L$-values with fixed modulus.
\begin{cor}
Let $q$ be a large prime and $\alpha> 0$ be any positive real number. Then for $K_\alpha$ as in Equation \eqref{ca},
\begin{equation}
\frac{1}{\varphi^+(q)} \sum_{\chi \text{ even primitive mod }q} \left|L\left(\frac{1}{2},\chi\right)\right|^{2\alpha} \gg  K_\alpha \left(\log q\right)^{\alpha^2}.
\end{equation}
\end{cor}

\subsection{Organisation of the proof}
The proof of Theorem \ref{LBLD} relies on the recursive scheme of \cite{ABR20}, as in  \cite{AB23} and \cite{AB24}.
The method approximates $\log |\zeta|$ by the real parts of Dirichlet polynomials using an appropriate mollification on the critical line.
In order to compute large moments of such polynomials, we need to restrict  the associated partial sums.
This is explained in detail in Section \ref{sect: walks}.
To attain a sharp lower bound on the critical line, we use the Paley-Zygmund inequality on mollified $\zeta$.
The bound then follows by computing twisted second and fourth moments, cf.~Proposition \ref{prop:goodprobs}.
This is the most technical part of the paper since to improve the constant it is necessary to work with a tilted probability measure.
Proposition \ref{prop:goodprobs} is proved in Section \ref{Sec: Mainproof}.
One ingredient is Proposition \ref{Prop:Twistmoll}, which gives precise estimates for moments of $\zeta$ twisted not by a Dirichlet polynomial, but rather a polynomial in the real part of a Dirichlet polynomial.
The proof of this auxiliary result is given in  Appendix \ref{Sec:Twistmoll}. Appendix \ref{Sec:Steinprob} gathers results on moments of the random Steinhaus model needed in the proof.

\subsection{Notation}
Throughout this paper, we assume that $t$ is uniformly distributed in the interval $[0,T]$ and write $\P$ and $\E$ for the corresponding expectation.

We make no effort to specify the absolute constant $C$ in Theorem \ref{LBLD}, and so for positive functions $f$ and $g$  we write  $f\ll_\alpha g,$ to denote that for  each fixed $\alpha$ there is a constant $C(\alpha)$  such that for all sufficiently large $T$, we have  $ f \le C(\alpha)g,$
and moreover there is an absolute constant $k,$ such that for all sufficiently large $\alpha,$  we have
\begin{equation}
\label{eqn: convention}
C(\alpha)\le \exp\left(k\alpha \right).
\end{equation}
We then use $f\asymp_\alpha g$ to denote $f\ll_\alpha g\ll_\alpha f.$

\stoptoc

\subsection{Acknowledgements} L.-P. Arguin is supported by the grants NSF DMS 2153803 and EPSRC EP/Z535990/1. N. Creighton is supported by EPSRC grant EP/W524311/1. N. Creighton would like to thank Madhurpana Das for her encouragement, and making available her thesis which contained results on Central Limit Theorems over a range of families of $L$-functions. Both authors would like to thank Jon Keating for his encouragement throughout the project, and Jad Hamdan for his thorough review of a previous draft. For the purpose of Open Access, the authors have applied a CC BY public copyright licence to any Author Accepted Manuscript (AAM) version arising from this
submission

\resumetoc

\section{Method of Proof}
\label{sect: method}

\subsection{The partial sums and their restrictions}
\label{sect: walks}
In this section, we explain the partial sums considered in \cite{AB23}, \cite{AB24} and \cite{AC25}.
In order to control the logarithm of $\zeta$, we define for $\eta>0$ the partial sums:
\begin{align}
S^{(1)}(\eta,t)=\sum_{\log\log p<\eta} \frac{1}{p^{{\frac{1}{2}+it}}},\qquad
S^{(2)}(\eta,t)=\sum_{\log\log p<\eta} \frac{1}{2p^{{1+2it}}}.
\end{align}
Then the combined approximation for $\log\zeta(\frac{1}{2}+it)$,  discarding the bounded contribution of powers of primes above squares, is given by:
\begin{equation}\label{eqn:Partialsum}
S(\eta,t)=S^{(1)}(\eta,t)+S^{(2)}(\eta,t).
\end{equation}
We expect the size of $S(\eta,t)$ to be largely influenced by $S^{(1)}(\eta,t).$
The  largest primes considered in the above sums will be up to
\begin{equation}
\label{eqn: T star}
T^\ast=T^{\frac{1}{\mathscr{R}(\alpha^2+1)}},
\end{equation}
for some large absolute constant $\mathscr{R}$. As will be explained in Section \ref{sect: proof thm}, the choice of $T^\ast$ is dictated by the  largest primes whose distribution we can control using the optimal choice for the degrees of tilts $k_{j,\mathbf{u}}$ in Equation \eqref{keqn}. We cannot show that, if the growth up to $p\le T^\ast$ obeys the linear growth at the log-log scale, then there is a chance of this persisting up to $p\le T$ for a large deviation of size $V$, since we have no control of the primes beyond $T^\ast.$ The proof instead shows that if the partial sum up to $p\le T^\ast$ is significantly larger than $V$, say

\begin{equation}\label{eqn: V tilde}
{V^*}=V+\mathscr{A}(\alpha+1)
\end{equation}  
for some large constant $\mathscr{A},$ then there is a large probability of the `full  sum' for $\log \left|\zeta\left(\frac{1}{2}+\ii t\right)\right|$ not falling below $V$. Because we have to truncate at $T^\ast$ rather than $T$, and aim for partial sums of size ${V^*}$ rather than $V$, we do not expect this to give a sharp value for the constant $K_\alpha$.

We divide the intervals of primes $p\leq T^*$ on the log-log scale at values $n_j,$ and study the partial sums on the intervals \begin{equation}\label{defPj}
\mathscr{P}_j=\{p \text{ prime: }\log\log p\in [n_{j-1},n_j)\}.
\end{equation}
The checkpoints $n_j$ are defined as follows.
Let $\mathscr{J}$ be the largest integer for which $\log_{\mathscr{J}+2}T>1000.$
For the scaling factor, $\mathbf{s}=10^6$, we set  for $1\le j\le \mathscr{J}$
\begin{equation}\label{eqn:logtime}
n_j=\log\log T^\ast-\mathbf{s}\left(\log_{j+2}T-\log_{\mathscr{J}+2}T\right), \quad T_j=e^{e^{n_j}}= T^{\frac{1}{\mathscr{R}(\alpha^2+1)} \frac{\log^{\mathbf{s}}_{\mathscr{J}+1}}{\log^{\mathbf{s}}_{j+1}}}.
\end{equation}
Note that $T_{\mathscr J}=T^*$.
For the small primes, we let $\mathscr{P}_0$ be the set of primes with $\log\log p$ between $$n_{-1}=-100,\quad \text{and}\quad n_{0}=\frac{\log_4 T}{1000}.$$
For the steps $n_j$ for $0\le j\le \mathscr{J},$ we define
\begin{equation}\label{def: Sj}
S_j(t)=\re(S(n_j,t)),
\end{equation}
to be the real parts of the partial sums defined in Equation \eqref{eqn:Partialsum} evaluated at the ends of the intervals.
These partial sums should behave like a random walk for $1\le j\le \mathscr{J}$. With this in mind, we will show that for most values of $t$ with $\log|\zeta(1/2+\ii t)|>V^*$, the partial sums $S_j(t)$ should lie in a small window around a line in the log-log scale for primes, with approximate gradient \begin{equation}\kappa^*=\frac{V^*}{n_{\mathscr{J}}}.\end{equation}  
 
More precisely, to control the small primes in $\mathscr{P}_0$, we take the mollifier
\begin{equation}
M_0(t)=\sum_{p\mid n\Rightarrow p\in \mathscr{P}_0}\frac{\mu(n)}{n^{\frac{1}{2}+\ii t}},
\end{equation}
and the restriction
\begin{equation}
G_0=\left\{\left|M_0(t)\right|^2\in [L_0,U_0]\right\},
\end{equation}  
with
\begin{equation}
L_0=\exp\left(-2\kappa^\ast n_0\right) \text{ and } U_0=\exp\left(-2\kappa^\ast n_0+n_0^{\frac{2}{3}}\right).
  \end{equation}
 These bounds represent an interval for  $|M_0(t)|^2\approx \exp\left(-2 S_0(t)\right)$ which has large probability given that $S_{\mathscr{J}}(t)>{V^*}$.
 
To get precise estimates, the restrictions for $S_j(t)$, $j\geq 1$, must depend on the value of $\left|M_0(t)\right|^2$.
We thus define the approximate gradient of the remainder of the walk to be
\begin{equation}
\label{eqn: kappa}
\kappa[z]=\frac{{V^*}-z}{n_{\mathscr{J}}-n_0},\  z>0.
\end{equation}
For  $1\le j\le \mathscr{J}$, the barrier for the partial sum at $n_j$ is then prescribed by
 \begin{equation}\label{eqn:barrier}
 \begin{aligned}
L_j(z)&=  z+\kappa[z] (n_j-n_{0})-10^4\log_{j+2} T,\\
U_j(z)&=z+ \kappa[z] (n_j-n_{0})+10^4\log_{j+2} T.
\end{aligned}
\end{equation}

In order to get a discrete set of intervals to sum over, we divide the interval $[L_0,U_0]$ into intervals of length $\Delta_0^{-1}$ where
\begin{equation}\label{d0} \Delta_0=e^{100\max \{\alpha,1\} n_0}.\end{equation}
Given a value of $\left|M_0(t)\right|^2$ lying in $[L_0,U_0]$, we define the event that the partial sum is good all the way up to $T^\ast$ as 
\begin{equation}\label{eqn:G}
G_{\mathscr{J}}= G_0\cap\left\{ S_j(t)\in [L_j\left(\mathfrak{n}\Delta_0^{-1}\right), U_j\left(\mathfrak{n}\Delta_0^{-1}\right)]\quad \forall 1\le j\le \mathscr{J},\quad \text{where } \mathfrak{n}=\lfloor -\Delta_0 \log \left|M_0(t)\right|\rfloor\right\}.
\end{equation}
We will henceforth omit the implicit dependence of $L_j(z)$ and $U_j(z)$ on the starting point $z$, since it is only at $n_1$ that there could be a significant deviation for large $T$.

For $1\le j\le \mathscr{J}$, we set the mollifiers to be: \begin{equation}
M_j(t)=\sum_{n} \frac{e_j(n)}{n^{\frac{1}{2}+it}},
\end{equation}
where $e_j(n)=\mu (n)
$ if $p|n\Rightarrow p\in \mathscr{P}_j$ and $\Omega_j(n)\le 5\alpha^2 (n_j-n_{j-1})^{10^5}$, and zero otherwise.
The complete mollifier is then \begin{equation}
\label{eqn: mollifier}
M(t)=\prod_{j=0}^{\mathscr{J}} M_j(t).
\end{equation}
The length of this Dirichlet polynomial is $\ll T^{\frac{1}{1000}}$ for a suitably large choice $\mathscr{R}$ in \eqref{eqn: T star}. It will transpire that for the small primes in $\mathscr{P}_0,$ we don't have strong Gaussian behaviour and so the twists must be long and the mollifier $M_0$ cannot be truncated. Conversely, for large $j$, the mollifier must be short to give a short twist, but this is permissible since the prime phases exhibit stronger Gaussian behaviour.
A tight control of the partial sums lying in the barrier corresponds to  pointwise bounds on the mollifier.
\begin{lem}\label{Mollbound}
For $t\in G_{\mathscr{J}}$,  we have  $\left|M(t)\right|^2\asymp_\alpha e^{-2L_{\mathscr{J}}}.$
\end{lem}

\begin{proof}
The proof follows from modifying the pointwise bounds on each section $M_j(t)$ of the mollifier from \cite[Lemma 23]{ABR20} conditional on the good event $G_{\mathscr{J}}.$ Taking the product of these pointwise bounds over all the intervals $\mathscr{P}_j$ for $0\le j\le \mathscr{J}$ completes the proof.
\end{proof}

\subsection{Proof of Theorem \ref{LBLD}}
\label{sect: proof thm}
We assume $\alpha\ge 1$. An easy adaptation of the method yields the bounds with $K_\alpha\asymp 1$ as $\alpha\to 0^+$.
From here on in, we remove the implicit dependence of $\zeta$ and $M$ on $t,$ to ease the notation.
Restricting to the event $G_{\mathscr{J}}$, we have the following lower bound on $\P(\log |\zeta|>V)$.
\begin{equation}
\label{eqn: proof1}
\P\left(\left\{\log |\zeta|>V\right\}\cap G_{\mathscr{J}}\right)=\P(|\zeta|^2>e^{2V}| G_{\mathscr{J}})\cdot \P(G_{\mathscr{J}})\geq \P(|\zeta M|^2>e^{2(V-L_{\mathscr{J}})}| G_{\mathscr{J}})\cdot \P(G_{\mathscr{J}}).
\end{equation}
A lower bound on the conditional probability can be attained by using the Paley-Zygmund inequality \eqref{eqn:Paley-Zygmund}:
\begin{equation}
\label{eqn: proof2}
\P(|\zeta M|^2>e^{2(V-L_{\mathscr{J}})}| G_{\mathscr{J}})\geq (1-\delta)^2 \frac{\E[|\zeta M|^2| G_{\mathscr{J}}]}{\E[|\zeta M|^4| G_{\mathscr{J}}]}^2,
\end{equation}
where 
\begin{equation} \label{eqn:newdelta}
	\delta=\frac{\exp({2V-2L_{\mathscr{J}}})}{\E[|\zeta M|^2| G_{\mathscr{J}}]}.
\end{equation}

By the definition of $T_{\mathscr{J}}$ and of the mollifiers, we expect that the event $G_{\mathscr{J}},$ which only depends on the primes $p\le T_{\mathscr{J}},$ will decorrelate from the mollified value  $\left|\zeta M\right|^2.$ Thus we should have that the above is
\begin{align}
&\left(1-\delta\right)^2  \frac{\E\left[\left|\zeta M\right|^2\right]^2}{\E\left[\left|\zeta M
\right|^4\right]}\asymp_\alpha \nonumber \frac{(\log T_{\mathscr{J}})^2}{(\log T)^2 } \left(1-\delta\right)^2 \gg_\alpha 1.
\end{align}
To ensure $\delta<1$, since  $L_{\mathscr{J}}={V^*}-O(1)$, we need to take $V^*$ slightly larger than $V$.

It remains to evaluate $\P(G_{\mathscr{J}})$. Here, we can resort to Gaussian approximation results for the distribution of the sum of the random variables at each prime for all the intervals $\mathscr{P}_j$ by approximating indicator functions by polynomials as in \cite{AB24}.  The resulting Gaussian approximation results  apply asymptotically for the intervals $\mathscr{P}_j,$ for $ j\ge 1,$ but for the small primes in $\mathscr{P}_0,$  we also pick up the arithmetic factor $a_\alpha$ from Equation \eqref{eqn: factors}. So we expect:
\begin{equation}
\P(G_{\mathscr{J}})\asymp_\alpha a_\alpha \frac{e^{-\frac{{V^*}^2}{\log\log T_{\mathscr{J}}}}}{\sqrt{\log\log T_{\mathscr{J}}}}\asymp_\alpha a_\alpha e^{-2\alpha^2\log \alpha +(2\mathscr{A}+4\gamma)\alpha^2 }\cdot \frac{e^{- \frac{V^2}{\log\log T}}}{\sqrt{\log\log T}}.
\end{equation}
However, we cannot pass directly to the random Steinhaus model using indicator functions $\mathbf{1}(G_{\mathscr{J}})$ as twists, and we must instead convert the restrictions into Dirichlet polynomials. The event $G_{\mathscr{J}}$ is very rare compared to a typical deviation of the size $\sqrt{\log\log T}$ from Selberg's Central Limit Theorem, thus the necessary Dirichlet polynomials required to detect a large deviation have very high degrees. This means we must truncate the partial sums at relatively small primes  to have a permissible twist where mean value theorems apply. In \cite{AB24}, this truncation occurred at $T^{1/\alpha^m}$ for a power $m$ much larger than $2$, which leads to a worse constant. However, because we tilt rather than relying solely on indicator functions uniformly close to $1$, we can use shorter polynomials, and so we are able to control primes all the way up to  $T_{\mathscr J}=T^{\frac{1}{\mathscr{R}(\alpha^2+1)}}$.
The following proposition, proved in the next section, gives bounds on $\P(G_{\mathscr{J}})$ and the necessary expectations.


\begin{prop}\label{prop:goodprobs}
For the event $G_{\mathscr{J}}$ in \eqref{eqn:G}, we have
\begin{equation}\label{notwisteqn}
\P(G_{\mathscr{J}})\asymp_\alpha(e^{-\gamma}\alpha^2\log\alpha)^{-\alpha^2} \frac{e^{-\frac{V^{*2}}{\log\log T}}}{\sqrt{\log\log T}},
\end{equation}
\begin{equation}\label{2twisteqn}
\E\left[\left|\zeta M\right|^2 \mathbf{1}\left(G_{\mathscr{J}}\right)\right]\asymp_\alpha   (e^{-\gamma}\alpha^2\log\alpha)^{-\alpha^2} \frac{e^{-\frac{V^{*2}}{\log\log T}}}{\sqrt{\log\log T}} \text{ and }
\end{equation}
\begin{equation}\label{4twisteqn}
\E\left[\left|\zeta M\right|^4\mathbf{1}\left(G_{\mathscr{J}}\right)\right]\ll_\alpha   (e^{-\gamma}\alpha^2\log\alpha)^{-\alpha^2} \frac{e^{-\frac{V^{*2}}{\log\log T}}}{\sqrt{\log\log T}}.
\end{equation}
\end{prop}
Substituting the expectations \eqref{notwisteqn} and \eqref{2twisteqn} into Equation \eqref{eqn:newdelta}, observing that $L_{\mathscr{J}}={V^*}-O(1)$, and recalling the convention \eqref{eqn: convention}, we see that
\begin{equation}
\delta\le  \exp\left(2V-2{V^*}+O\left(\alpha\right)\right).
\end{equation}
Hence, it suffices to choose ${V^*}=V+O(\alpha)$ as in \eqref{eqn: V tilde} to ensure that $\delta\le e^{-1}$.
This completes the proof of Theorem \ref{LBLD}.

\section{Proof of Proposition \ref{prop:goodprobs}}\label{Sec: Mainproof}
The proof of Proposition \ref{prop:goodprobs} is based on approximating indicator functions for the good event $G_{\mathscr{J}}$ by polynomials evaluated at real parts of Dirichlet polynomials, and on twisted moments of these quantities.
We review these tools in Sections \ref{sect: approx G} and \ref{sect: twisted real}.

\subsection{Approximation of the events $G_{\mathscr{J}}$}
\label{sect: approx G}

We partition the intervals $[L_j,U_j]$ into  subintervals of length $\Delta_j^{-1}$, where
\begin{equation}\label{eqn:delj}
\Delta_j= \left\lfloor \log^{10}_{j+1}T\right\rfloor.
\end{equation}
We will see from Equation \eqref{seplen} that $\Delta_j$ cannot be taken much larger. The key observation is that $\Delta_j$ does not grow with $\alpha$, enabling us to keep the Dirichlet polynomials for the indicator functions short.

For a lower bound, we need to split the region within the barriers into the union of intervals contained in $[L_j,U_j]$ for each $j$, whilst for an upper bound, we need to ensure the union covers $[L_j,U_j]$. We thus define the sets $\mathfrak{I}^+$ for the upper bound, and $\mathfrak{I}^-$ for the lower bound. We take $\mathfrak{I}^{\pm}$ to be the set of tuples $\left(u_0,...,u_{\mathscr{J}}\right)$ such that: \begin{itemize}
\item $u_j\Delta_j\in \Z  $ for each $j$,\\
\item $u_0\in [L_0\mp \Delta_0^{-1},U_0\pm \Delta_0^{-1}]$  and \\
\item$-\frac{1}{2}\log u_0+\sum^j_{m=1}u_m\in [L_j\mp \Delta_j^{-1},U_j\pm \Delta_j^{-1}]\ \forall 1\le j\le \mathscr{J}.$
\end{itemize}
We then need to understand the increments
\begin{equation}
\label{eqn: Yj}
Y_j=S_j-S_{j-1}, \qquad 0\leq j\leq \mathscr{J}.
\end{equation}
 We partition on the values of the quantities \begin{equation}
Z_j=\begin{cases} |M_0|^2& j=0,\\
Y_j & 1\le j \le \mathscr{J}.\end{cases}
\end{equation}

We relate the values $Z_j$  to the partial sums $S_j$ as follows:
\begin{align}\label{reverse}
\cup_{\mathbf{u}\in \mathfrak{I}^-} \{Z_j\in [u_j, u_j+\Delta_j^{-1}] \forall j\} \subset\left\{ |M_0|^2\in [L_0, U_0+\Delta_0^{-1}],\ S_{j}\in [L_j, U_j+\Delta_j^{-1}],\ 1\le j\le \mathscr{J} \right\},
\end{align}
and
\begin{align}\label{includes}
\cup_{\mathbf{u}\in \mathfrak{I}^+} \{Z_j\in [u_j, u_j+\Delta_j^{-1}] \forall j\} \supset\{ |M_0|^2\in [L_0+\Delta_0^{-1}, U_0],\ S_{j}\in [L_j+\Delta_j^{-1}, U_j], \ 1\le j\le \mathscr{J} \}.
\end{align}
The events on the left-hand side are essentially disjoint for different tuples $\mathbf{u}\in \mathfrak{I}^{\pm}.$ We hence observe that \begin{align}
		\nonumber\sum_{\mathbf{u}\in \mathfrak{I}^-} &\E\left[\left|\zeta M\right|^2 \mathbf{1}(Z_j\in [u_j, u_j+\Delta_j^{-1}]\forall 0\le j\le \mathscr{J}) \right]\\\le &\E\left[\left|\zeta M\right|^2 \mathbf{1}(G_{\mathscr{J}})\right]\\\nonumber\le \sum_{\mathbf{u}\in \mathfrak{I}^+} &\E\left[\left|\zeta M\right|^2 \mathbf{1}(Z_j\in [u_j, u_j+\Delta_j^{-1}]\forall 0\le j\le \mathscr{J})\right].
	\end{align} 
In order to bound each expectation it is necessary to convert the indicator functions into Dirichlet polynomials. First, we must tilt to reduce the necessary length of the Dirichlet polynomials to give a twist which can be handled by Proposition \ref{Prop:Twistmoll}.

A standard approach in large deviations is to introduce an exponential  tilt for values to become typical in the tilted measure.
To improve the constant, we use a polynomial  tilt, which is better suited to applying the Mean Value Theorem for Dirichlet polynomials.
The tilt effectively reduces the degree required for the Dirichlet polynomials mimicking the indicator functions.
	
We begin by calculating the lower bound for each $u\in \mathfrak{I}^-.$ Instead of calculating\\ $\E\left[\left|\zeta M\right|^2 \mathbf{1}(Z_j\in [u_j, u_j+\Delta_j^{-1}]\forall 0\le j\le \mathscr{J})\right],$ we multiply by the partial sums over the intervals $\mathscr{P}_j,$ each raised to a power depending on the values of:
\begin{equation}\label{eqn:z_j}
	z_{-1}=0,\quad z_j=-\frac{1}{2}\log u_0+\sum^{j}_{k=1} u_k, \quad 0\le j\le \mathscr{J}.
\end{equation}
We consider the  tilted expectations $\E\left[\left|\zeta M\right|^2 \prod^{\mathscr{J}}_{j=0} Y_j^{2k_{j,\mathbf{u}}} \mathbf{1}(Z_j\in [u_j+\Delta_j^{-a}, u_j+\Delta_j^{-1}])\right],$
	where
\begin{equation}\label{keqn}
	k_{j,\mathbf{u}}=\lfloor \kappa^*(z_j-z_{j-1})\rfloor,\ \text{ for $0\le j\le \mathscr{J}.$}
\end{equation}
Thus, $k_{j,\mathbf{u}}=\lfloor \kappa^*u_j\rfloor$ for $1\leq j\leq \mathcal J$, and $k_{0,\mathbf{u}}=\lfloor -\frac{\kappa^*}{2}\log u_0\rfloor$.
If we have values close to a given target interval $\mathfrak{u}\in \mathfrak{I}^{\pm}$ then the following key lemma shows we can specify the value of the tilt up to $\asymp_\alpha $ factors.
 \begin{lem}\label{Lem:MAGIC}
	Let $\mathbf{u}\in \mathfrak{I}^\pm$ and $c_{-1},c_0,...,c_{\mathscr{J}}$ be real numbers satisfying \begin{itemize}
		\item $c_{-1}=0$
		\item  $\left|c_j-z_j\right|\le 100 $ for each $0\le j\le \mathscr{J}.$

	\end{itemize}
	Then $\prod^{\mathscr{J}}_{j=0} \left(c_j-c_{j-1}\right)^{2k_{j,\mathbf{u}}}\asymp_\alpha \mathfrak{c}_{\mathbf{u}},$ where

	\begin{equation}
		\mathfrak{c}_{\mathbf{u}}=\prod^{\mathscr{J}}_{j=0} \left(z_j-z_{j-1}\right)^{2k_{j,\mathbf{u}}}.
	\end{equation}
\end{lem}

\begin{proof}
	By Taylor expansion of the moments, we see \begin{equation}
		\prod^{\mathscr{J}}_{j=0}\left(\frac{c_j-c_{j-1}}{z_j-z_{j-1}}\right)^{2k_{j,\mathbf{u}}}=\exp\left(\sum^{\mathscr{J}}_{j=0}2k_{j,\mathbf{u}}\left(\frac{c_j-z_j-(c_{j-1}-z_{j-1})}{z_j-z_{j-1}}+O\left(\frac{1}{\left(z_j-z_{j-1}\right)^2}\right)\right)\right).
	\end{equation}
	
	By Equation \eqref{keqn}, we see the total error term is
	\begin{equation}
		\exp\left(\sum^{\mathscr{J}}_{j=1}\frac{2k_{j,\mathbf{u}}}{(z_j-z_{j-1})^2}\right)=\exp(O(1)), 
	\end{equation}
	and hence we may bound the above quotient as
	
	\begin{equation}
		\exp\left(O(1)+\sum^{\mathscr{J}-1}_{j=0}(c_j-z_{j})\left(\frac{2k_{j,\mathbf{u}}}{z_j-z_{j-1}}-\frac{2k_{j+1,\mathbf{u}}}{z_{j+1}-z_{j}}\right)+\frac{2k_{\mathscr{J},\mathbf{u}}\left(c_{\mathscr{J}}-z_{\mathscr{J}}\right)}{z_{\mathscr{J}}-z_{\mathscr{J}-1}}\right).
	\end{equation}
	By Equation \eqref{keqn} we have $k_{j,\mathbf{u}}=\kappa^*(z_j-z_{j-1})+O(1)$ for $0\le j\le \mathscr{J}$ and so this is 
	\begin{equation}
		\exp\left(O(1)+2\kappa^*\left(c_{\mathscr{J}}-z_{\mathscr{J}}\right)+\sum^{\mathscr{J}-1}_{j=0}O\left(\frac{1}{z_j-z_{j-1}}+\frac{1}{z_{j+1}-z_{j}}\right)\right)=\exp(O(1+\kappa^*)).
	\end{equation}
	This is acceptable as an $\asymp_\alpha $ factor, since $\kappa^*\to \alpha$ as $T\to\infty.$  
\end{proof}
The indicator function for each interval can be approximated by a polynomial  due to the following lemma from \cite{AB24}.
\begin{lem}[Lemma 2 in \cite{AB24}]\label{Lem: indicators}
For all $X>100$, $\Delta$ sufficiently large and $a>2$, there exist polynomials  $\mathscr{D}^+$ and $\mathscr{D}^-$ of degree smaller than $100X\Delta^{3a}$ with $l^{\text{th}}$ coefficient bounded by $\frac{(2\pi)^l}{l!}\Delta^{2a(l+2)}$ such that for $|x|\le X$, we have

\begin{equation}
\mathbf{1}\left(x\in \left[0,\Delta^{-1}\right]\right)(1-e^{-\Delta^{a-2}})\le |\mathscr{D}^+(x)|^2\le \mathbf{1}\left(x\in \left[-\Delta^{-a},\Delta^{-1}+\Delta^{-a}\right]\right)+e^{-\Delta^{a-2}}.
\end{equation}  
and \begin{equation}\label{negbound}
\mathbf{1}\left(x\in \left[\Delta^{-a},\Delta^{-1}-\Delta^{-a}\right]\right)(1-e^{-\Delta^{a-2}})\le |\mathscr{D}^-(x)|^2\le \mathbf{1}\left(x\in \left[0,\Delta^{-1}\right]\right)+e^{-\Delta^{a-2}}.
\end{equation}
\end{lem}
The choice of $a$, provided it is fixed, will not have any qualitative effect in the proof; at most it will effect the constant $C$ in Equation \eqref{ca}. We will take $a=5$ throughout.  

In \cite{AB24}, an accurate upper bound on the twisted moments was not required to give a lower bound on the probability and so the bound from the Cauchy-Schwarz inequality sufficed. This is no longer the case when tilting to recover a constant of the shape in Equation \eqref{ca}. Thus, it is necessary to control the contribution of points far away from the interval $[L_j,U_j]$, which translates into a bound on the growth of $\mathscr{D}^{\pm}(y)$ when $|y|$ is large. The following simple lemma gives a  sufficient bound on $\mathscr{D}^\pm(y)$ when $|y|\ge X$.
\begin{lem}\label{Lem: Dbound}
If $X$ and $\Delta$ are sufficiently large and $|y|\ge X,$ then \begin{equation}\left|\mathscr{D}^\pm(y)\right|\le\left|2y/X\right|^{100 X \Delta^{3a}}.\end{equation}
\end{lem}
\begin{proof} Substituting the bounds on the coefficients  of $\mathscr{D}^\pm$ from Lemma \ref{Lem: indicators} shows that  \begin{align}
\left|\mathscr{D}^\pm(y)\right|&\le \sum^{100 X \Delta^{3a}}_{l=0} \frac{(2\pi)^l}{l!}\Delta^{2a(l+2)}|y|^l
= \sum^{100 X \Delta^{3a}}_{l=0}\nonumber \frac{(2\pi)^l}{l!}\Delta^{2a(l+2)}|X|^l\left|\frac{y}{X}\right|^l.\end{align}
Since $|y|\ge X$ we can  bound the last term  to write the above as
\begin{equation}\le\left|\frac{y}{X}\right|^{100 X \Delta^{3a}}\Delta^{4a} \sum^{100 X \Delta^{3a}}_{l=0} \frac{(2\pi)^l}{l!}\Delta^{2al}X^l\le \left|\frac{y}{X}\right|^{100 X \Delta^{3a}}\Delta^{4a} \exp(2\pi \Delta^{2a} X),\end{equation} which follows from truncating the exponential Taylor series.
For $X$ and $\Delta$ sufficiently large, this is $\le \left|2y/X\right|^{100 X \Delta^{3a}}.$
\end{proof}

\subsection{Twisted moments for real partial sums}
\label{sect: twisted real}
In this section, we follow \cite{C25}.
Motivated by the idea of showing that the behaviour of the partial sums over different intervals $\mathscr{P}_j$ should decorrelate, we define $\alpha$-separable twists. The number of prime factors required in each interval grows with $\alpha$, and so we must truncate at smaller primes to keep the total length of the Dirichlet polynomial short.
The following definition will apply to the product of the polynomials $\mathscr{D}^\pm$ evaluated at the real partial sums $Y_j$.
\begin{df}
\label{df: Q}
We say $Q$ is $\alpha$-separable if it is of the form
\begin{equation} Q(t)=\sum_{\substack{p|mn\Rightarrow \log\log p\le n_{\mathscr{J}}\\ \Omega_0(mn)\le \log^{\frac{1}{100}}T\\ \Omega_j(mn)\le  \alpha^2(n_j-n_{j-1})^{10^4}\forall 1\le j\le \mathscr{J}}}
\frac{a(n,m)}{n^{\frac{1}{2}+it}m^{\frac{1}{2}-it}},\end{equation}
where $a(n,m)=\overline{a(m,n)}$ for all $n,m$.  Here, $\Omega_l(n)$ denotes the number of prime factors of $n$ (with multiplicity) that lie in $\mathscr{P}_l$.
\end{df}
The condition $a(n,m)=\overline{a(m,n)}$ ensures that $Q(t)$ is real whenever $t$ is real. Observe moreover that an $\alpha$-separable polynomial can have length at most:
\begin{equation}\label{seplen}e^{e^{n_0}\log^{\frac{1}{100}}T}\cdot \prod^{\mathscr{J}}_{j=1}T_j^{\alpha^2(n_j-n_{j-1})^{10^4}}=T^{\sum^{\mathscr{J}}_{j=1}\frac{\alpha^2(n_j-n_{j-1})^{10^4}\log_{\mathscr{J}+1}^sT}{\mathscr{R}(\alpha^2+1)\log_{j+1}^sT}+o(1)}\le T^{\frac{1}{1000}},\end{equation}
for a suitably large absolute constant $\mathscr{R},$ which means we can apply the Mean Value Theorem for Dirichlet Polynomials to evaluate $\E\left[Q^2\right].$
In \cite{C25}, the distribution of quadratic twists of elliptic curves are considered; there may be much cancellation between terms with the same square-free parts, and so in order to preserve these in relating twisted mollified moments, these terms are grouped together. Inspired by this, we group together terms in $Q^2$, and using the relation $(rm)^{it}(rn)^{-it}=m^{it}n^{-it},$ we may write

\begin{equation}\label{Q2}
Q^2=\sum_{n,m} \frac{b(n,m)}{m^{\frac{1}{2}+it}n^{\frac{1}{2}-it}},
\end{equation}
for some coefficients $b(n,m)$ such that $b(n,m)=\overline{b(m,n)}$ and $b(n,m)=0$ if $(n,m)>1$. \\
Indeed, if we expand out $Q^2$, we get
\begin{equation}\label{unsimpleeqn}
Q^2=\sum_{u,v} \frac{P(u,v)}{u^{\frac{1}{2}+it}v^{\frac{1}{2}-it}},\qquad P(u,v)=\sum_{\substack{n_1n_2=u\\ m_1m_2=v}} a(n_1,m_1)a(n_2,m_2).
\end{equation}
If we take out any common factors of $u$ and $v$ in Equation \eqref{unsimpleeqn}, we then have
\begin{equation}\label{eqn:bcoef}
b(n,m)=\sum_{r} \frac{P(rn,rm)}{r} \qquad \text{if $(n,m)=1$, and $0$ otherwise.}
\end{equation}

Note that $|b(n,m)|\le b(1,1)$ for all $n,m$, thus the integral  $\frac{1}{T}\int^{T}_0 Q(t)^2\rd t$ is dominated by the diagonal term $b(1,1)$.
Indeed, when $(u,v)=1$ we may set $
\mathcal{C}(u,v)=\sum_r a(ru,rv),
$  
so that
\\
\begin{align}
|b(n,m)|=\Big|\sum_{\substack{(u_1,v_1)=(u_2,v_2)=1\\ nv_1v_2=mu_1u_2}} \mathcal{C}(u_1,v_1)\mathcal{C}(u_2,v_2)\Big|
\le \sum_{\substack{(u_1,v_1)=(u_2,v_2)=1\\ nv_1v_2=mu_1u_2}} \frac{1}{2}\left(\left|\mathcal{C}(u_1,v_1)\right|^2+\left|\mathcal{C}(u_2,v_2)\right|^2\right).
\end{align}

Since the sum is supported on coprime integers, each pair $(u,v)$ appears at most once as $(u_1,v_1)$ and at most once as $(u_2,v_2)$, and so the above sum is at most $b(1,1).$
With these considerations, the twisted moment formulae proved in Appendix \ref{Sec:Twistmoll} take the following form.
\begin{prop}
\label{Prop:Twistmoll}
Let $Q$ be $\alpha$-separable and $M$ be as in Equation \eqref{eqn: mollifier}. Then with the coefficients $b(c,g)$ as defined in Equation \eqref{eqn:bcoef}, we have
\begin{align}
\label{2}
&\int^T_0\Big|\zeta\left(\frac{1}{2}+it\right) M(t)\Big|^2 Q(t)^2\rd t\asymp T\frac{\log T}{\log T_{\mathscr{J}}} b(1,1),\\
\label{4}
&\int^T_0\Big|\zeta\left(\frac{1}{2}+it\right)  M(t)\Big|^{4} Q(t)^2\rd t\ll T\left(\frac{\log T}{\log T_{\mathscr{J}}}\right)^4 b(1,1),
  \end{align}
where the implicit constants are independent of $\alpha.$
\end{prop}

\subsection{Estimating restricted expectations}
We focus on the proof of \eqref{2twisteqn}, then explain the modifications needed to prove Equations \eqref{notwisteqn} and \eqref{4twisteqn} at the end of the section. We use the inclusions  \eqref{reverse}-\eqref{includes} and the polynomials $\mathscr{D}^+$ and $\mathscr{D}^-$ from Lemma \ref{Lem: indicators}  to approximate the indicator event $G_{\mathscr{J}}$. For the interval $\mathscr{P}_0,$ because the primes are sufficiently small, we can approximate the indicator function to the target interval for any value of $\left|M_0\right|^2$ up to the pointwise bound, and so we take  \begin{equation}X_0=\prod_{p\in \mathscr{P}_0} \left(1-\frac{1}{p^{\frac{1}{2}}}\right)^{-2}=(\log T)^{o(1)}.\end{equation}

For $1\leq j\leq\mathscr{J}$, we approximate the indicator function to the interval of length $\Delta_j^{-1}$ up to magnitudes of size
\begin{equation}\label{chooseXeqn}
X_j=\lceil100\left(\alpha(n_j-n_{j-1})+10^4\log_{j+1}T\right)\rceil.
\end{equation}
Writing $\mathscr{D}^\pm_j$ for the polynomials with this choice of polynomials for each $j$, we now see that, for any choice of partitions $\mathbf{u}\in \mathfrak{I}^{\pm},$ we have:
\begin{align}
\nonumber\E&\left[ \left|\zeta M\right|^2 Y_0^{2k_{0,\mathbf{u}}}\left(\mathscr{D}_0^-(\left|M_0\right|^2-u_0)^2+\epsilon_0\right)\prod_{1\le j\le \mathscr{J}}Y_j^{2k_{j,\mathbf{u}}} \left(\mathscr{D}^-_j(Y_j-u_j)^2+\epsilon_j\right)\mathbf{1}\left(|Y_j-u_j|\le X_j\right)\right]\\\label{eqn:boundindic}&\le\E\left[\left|\zeta M\right|^2 \prod^{\mathscr{J}}_{j=0}Y_j^{2k_{j,\mathbf{u}}}  \mathbf{1}\left(Y_j\in \left[u_j,u_j+\Delta_j^{-1}\right]\right)\right]\le\\ \E&\left[ \left|\zeta M\right|^2 Y_0^{2k_{0,\mathbf{u}}}\left(\mathscr{D}^+_0(\left|M_0\right|^2-u_0)^2+\epsilon_0\right)\prod_{1\le j\le \mathscr{J}}Y_j^{2k_{j,\mathbf{u}}} \left(\mathscr{D}^+_j(Y_j-u_j)^2+\epsilon_j\right)\mathbf{1}\left(|Y_j-u_j|\le X_j\right)\right].\nonumber
\end{align}

For convenience, we have set $\epsilon_j=O(e^{-\Delta_j})$
to be the error term from truncating the approximation of the indicator functions, which may change from line to line.
We now prove:
\begin{equation}
\label{eqn: LB to prove}
\begin{aligned}
& \sum_{\mathbf{u}\in \mathfrak{I}^-} \frac{1}{\mathfrak{c}_{\mathbf{u}}}\E\Big[\left|\zeta M\right|^2 Y_0^{2k_{0,\mathbf{u}}}\left(\left(\mathscr{D}^-_0(|M_0|^2-u_0)\right)^2+\epsilon_0\right)\prod_{1\le j\le \mathscr{J}}Y_j^{2k_{j,\mathbf{u}}} \left(\mathscr{D}^-_j(Y_j-u_j)^2+\epsilon_j\right)\mathbf{1}\left(|Y_j-u_j|\le X_j\right)\Big]\\&\nonumber
\gg_\alpha a_\alpha \frac{e^{-\frac{V^{*2}}{\log\log T}}}{\sqrt{\log\log T}}.
\end{aligned}
\end{equation}
A matching upper bound with $\mathscr{D}^+$ on $\mathfrak{I}^+$ follows the same way, and we omit the details. \\

For the first interval comprised of the primes $p\in \mathscr{P}_0,$ we have taken $X_0$ sufficiently large that we have a pointwise bound $|M_0|^2\le X_0$, and so the Dirichlet polynomial $\mathscr{D}_0^-(|M_0|^2-u_0)$ is always a good approximation for the indicator function $\mathbf{1}\left(|M_0|^2\in [u_0, u_0+\Delta_0^{-1}]\right).$ For the intervals $\mathscr{P}_j$ with $j\ge 1$, we were unable to take $X_j$ sufficiently large for this pointwise bound, and maintain a short enough Dirichlet polynomial to use Proposition  \ref{Prop:Twistmoll}. However, with the aid of the tilt, we will show that this interval still dominates the expectation. We can use the Inclusion-Exclusion Principle  to rewrite the first line in Equation \eqref{eqn:boundindic} as:
\begin{equation}
\sum_{R\subset\{1,...,\mathscr{J}\}} (-1)^{|R|} \E\left[ \left|\zeta M\right|^2 \prod_{0\le j\le \mathscr{J}} \left(\mathscr{D}^-_j(Y_j-u_j)^2+\epsilon_j\right)\prod_{r\in R}\mathbf{1}\left(|Y_r-u_r|> X_r\right)\right].
\end{equation}
The dominant term comes from $R=\emptyset$, when none of the partial sums  $Y_r$ are too far from the target intervals. When $R\neq \emptyset$, we use Markov's inequality on each interval $\mathscr{P}_j$ with $j\in R$ to express the above as
\begin{align}\nonumber
&\E\Big[ \left|\zeta M\right|^2 \prod_{0\le j\le \mathscr{J}} \Big(\mathscr{D}^-_j(Y_j-u_j)^2+\epsilon_j\Big)\Big]+\\\sum_{\substack{R\subset\{1,...,\mathscr{J}\}\\ R\ne \emptyset}}O&\Big(   \E\Big[ \left|\zeta M\right|^2\prod_{0\le j\le \mathscr{J}} \left(\mathscr{D}^-_j(Y_j-u_j)^2+\epsilon_j\right)\prod_{r\in R}\frac{|Y_r-u_r|^{2X_r}}{X_r^{2X_r}}\Big]\Big) .
\label{eqn:Inc-Exc}
\end{align}
We now estimate the above expectations by using Proposition \ref{Prop:Twistmoll}.

We need to show the twist, after we fully expand the product, is $\alpha$-separable in the sense of Definition \ref{df: Q}.
The summand with the most prime factors is from $R=\{1,...,\mathscr{J}\}$. The twist then amounts to
$$
\left(S_1-S_0\right)^{2k_{0,\mathbf{u}}}\mathscr{D}^-_0\left(Y_0-u_0\right)^2\prod^{\mathscr{J}}_{j=1}Y_j^{2k_{j,\mathbf{u}}}\mathscr{D}^-_j\left(Y_j-u_j\right)^2(Y_j-u_j)^{2k_{j}}.$$
The terms in the twist have the number of prime factors in $\mathscr{P}_0$ given by\\ $\Omega_0(mn)\le 200X_0 \Delta_0^{3a}+2k_{0,\mathbf{u}}\le \log^{\frac{1}{100}}T,$ while for $1\le j\le \mathscr{J},$ $\Omega_j(mn)$ is bounded above by
$$
400X_j\Delta_j^{3a}+2k_{j,\mathbf{u}}\le 600(100\alpha(n_j-n_{j-1}+10^5\log_{j+1}T)\log^{30a}_{j+1}T+4\alpha^2(n_j-n_{j-1}))\le \alpha^2(n_j-n_{j-1})^{10^4},
$$
for $T$ sufficiently large. Hence $Q$ is $\alpha$-separable, which allows us to take the twisted moments.

The Mean Value Theorem for Dirichlet Polynomials and its twisted mollified analogue in Proposition \ref{Prop:Twistmoll} effectively replace $p^{-it}$ by independent Steinhaus random variables in \eqref{def: Sj}. More precisely, if $\theta=(\theta_p, p \text{ primes})$ are independent and identically distributed uniformly on $\R/2\pi\Z$, then the analogue of $Y_j$ in \eqref{eqn: Yj}, $j\ge 0
$, in the Steinhaus model is
\begin{equation}\label{eqn: MathscrYj}
\mathscr{Y}_j=\sum_{p\in \mathscr{P}_j} \frac{\cos \theta_p}{p^{\frac{1}{2}}}+\frac{\cos 2\theta_p}{2p}.
\end{equation}
	The variance of this increment is exactly 
\begin{equation}
	\label{eqn:Sigmaj}
	\sigma_j^2=\sum_{ p\in \mathscr{P}_j } \frac{1}{2p}+\frac{1}{8p^2}.
\end{equation}

To get an accurate bound for the first interval of primes $\mathscr{P}_0$, we require the full Euler product that includes the larger power of $p$:
\begin{equation}
\mathscr{M}_{0}=\prod_{p\in \mathscr{P}_0}1-\frac{e^{i\theta_p}}{p^{1/2}}.
\end{equation}
For a given $\alpha$-separable polynomial $Q(t)$ in Definition \ref{df: Q}, we have the corresponding multinomial in the Steinhaus model
\begin{equation}\label{random}
\mathcal{Q}(\theta)=\sum_{m,n}\frac{a(n,m)}{n^{\frac{1}{2}}m^{\frac{1}{2}}}e^{i\theta_n}e^{-i\theta_m},
\end{equation}
where $\theta_n$ is uniquely defined  completely additively from the values of $\theta_p$ and the prime factorisation of $n$.
The expectation of the second moment is then given by the diagonal term in Equation \eqref{eqn:bcoef}: \begin{equation}\label{eqn:SteinExp}
\E\left[\mathcal Q(\theta)^2\right]=b(1,1).
\end{equation}

In light of the twisted mollifier formulae in Proposition \ref{Prop:Twistmoll}, we see that the twisted mollified moments are also dominated by this diagonal term. Comparing the diagonal terms $b(1,1)$ in Equation \eqref{2} for the second twisted mollified moment and in Equation \eqref{eqn:SteinExp} for the Steinhaus model, we see that Equation \eqref{eqn:Inc-Exc} is
\begin{align}
& \asymp \frac{\log T}{\log T_{\mathscr{J}}}\left( \E\left[\mathscr{Y}_0^{2k_{0,\mathbf{u}}} \left(\mathscr{D}_0^-(\left|\mathscr{M}_0\right|^2-u_0)^2+\epsilon_0\right)\prod^{\mathscr{J}}_{j=1} \mathscr{Y}_j^{2k_{j,\mathbf{u}}} \left(\mathscr{D}_j^-(\mathscr{Y}_j-u_j)^2+\epsilon_j\right) \right]+\right.\\\sum_{\substack{R\subset\{1,...,\mathscr{J}\}\\
R\ne \emptyset}}&O\left.\left(\E\left[\mathscr{Y}_0^{2k_{0,\mathbf{u}}} \left(\mathscr{D}_0^-(\left|\mathscr{M}_0\right|^2-u_0)^2+\epsilon_0\right)\prod^{\mathscr{J}}_{j=1} \mathscr{Y}_j^{2k_{j,\mathbf{u}}} \left(\mathscr{D}_j^-(\mathscr{Y}_j-u_j)^2+\epsilon_j\right)\prod_{r\in R}\frac{(\mathscr{Y}_r-u_r)^{2X_r}}{X_r^{2X_r}}\right]\right)\right).
\nonumber
\end{align}
Using the independence of the random variables associated to primes in the disjoint intervals, we may express the sum as:
\begin{align}
\label{eqn: steinhaus expect}
\asymp &\frac{\log T}{\log T_{\mathscr{J}}} \E\left[\mathscr{Y}_0^{2k_{0,\mathbf{u}}} \left(\mathscr{D}_0^-(\left|\mathscr{M}_0\right|^2-u_0)^2+\epsilon_0\right)\right]\times\\&\prod^{\mathscr{J}}_{j=1} \left(\E\left[\mathscr{Y}_j^{2k_{j,\mathbf{u}}} \left(\mathscr{D}_j^-(\mathscr{Y}_j-u_j)^2+\epsilon_j\right) \right]+O\left(\E\left[\mathscr{Y}_j^{2k_{j,\mathbf{u}}} \left(\mathscr{D}_j^-(\mathscr{Y}_j-u_j)^2+\epsilon_j\right)\frac{(\mathscr{Y}_j-u_j)^{2X_j}}{X_j^{2X_j}}\right]\right)\right).\nonumber
\end{align}

It now remains to bound each of these expectations.
We first calculate the first term involving the small primes.
\begin{prop}\label{Prop:Steinhausprobsinit}
The expectation associated with the first interval  of primes, $\mathscr{P}_0,$ is bounded below as follows:
\begin{equation}\label{firstprob}\E\left[\mathscr{Y}_0^{2k_{0,\mathbf{u}}} \left(\mathscr{D}_0^-(\left|\mathscr{M}_0\right|^2-u_0)^2+\epsilon_0\right)\right]\ge\E\left[\mathscr{Y}_{0}^{2k_{0,\mathbf{u}}}\mathbf{1}(\left|\mathscr{M}_0\right|^2 \in [u_0, u_0+\Delta_0^{-1}])\right]\exp\left(O\left(\Delta_0^{-\frac{1}{4}}\right)\right).\end{equation}
\end{prop}
\begin{proof}[Proof of Proposition \ref{Prop:Steinhausprobsinit}]
We follow the proof in Section 3.2 of \cite{AB24}, where many of the upper bounds can be turned into asymptotics.
Because $\left|\mathscr{M}_0\right|^2\le X_0$ pointwise, the polynomials in Lemma \ref{Lem: indicators} approximate the indicator functions for the entire range.
Equation \eqref{negbound} now yields that:
\begin{align}\nonumber&
\E\left[\mathscr{Y}_0^{2k_{0,\mathbf{u}}} \left(\mathscr{D}_0^-(\left|\mathscr{M}_0\right|^2-u_0)^2+\epsilon_0\right)\right]\ge \left(1-e^{-\Delta_0^{a-2}}\right)\times\\&\E\left[\mathscr{Y}_0^{2k_{0,\mathbf{u}}} \left(\mathbf{1}(\left|\mathscr{M}_0\right|^2\in [u_0+\Delta_0^{-a},u_0+\Delta_0^{-1}-\Delta_0^{-a}])+\epsilon_0\right)\right],
\end{align}
upon readjusting the error term $ \epsilon_0=O\left(e^{-\Delta_0}\right)$.
We observe that, following from Equation (67) in \cite{AB24}, and recalling from Equation \eqref{d0} that $\Delta_0=e^{100\alpha n_0}$, one has  \begin{equation}
\P(\left|\mathscr{M}_0\right|^2\in [u_0, u_0+\Delta_0^{-\frac{3}{2}}])\ge \Delta_0^{-\frac{3}{2}}\exp(-2n_0) \gg  \Delta_0 e^{-\Delta_0^{a-2}},
\end{equation}
and hence instead of Equation (68), we attain: \begin{align}
&\E\left[\left(\mathscr{D}_0^-(\left|\mathscr{M}_0\right|^2-u_0)^2+\epsilon_0\right)\mathscr{Y}_{0}^{2k_{0,\mathbf{u}}}\right]=\\&\nonumber\exp\left(O\left(\Delta_0^{-\frac{1}{2}}\left(\E\left[\mathscr{Y}_0^{2k_{0,\mathbf{u}}}\right]\right)\right)\right)\E\left[\mathscr{Y}_{0}^{2k_{0,\mathbf{u}}}\mathbf{1}(\left|\mathscr{M}_0\right|^2 \in [u_0+\Delta_0^{-1}, u_0+2\Delta_0^{-1}])\right].
\end{align}
Using  \eqref{totalmoment} to bound the moment $\E[\mathscr{Y}_0^{2k_{0,\mathbf{u}}}],$ we see that
\begin{align}
&\E\left[\left(\mathscr{D}_0^-(\left|\mathscr{M}_0\right|^2-u_0)^2+\epsilon_0\right)\mathscr{Y}_{0}^{2k_{0,\mathbf{u}}}\right]=\exp\left(O\left(\Delta_0^{-\frac{1}{4}}\right)\right)\E\left[\mathscr{Y}_{0}^{2k_{0,\mathbf{u}}}\mathbf{1}(\left|\mathscr{M}_0\right|^2 \in [u_0, u_0+\Delta_0^{-1}])\right]).
\end{align}
\end{proof}
The expectations in \eqref{eqn: steinhaus expect} associated to the approximations $\mathscr{D}_j^{\pm}$ to the indicator functions for partial sums up to $p\le T_j$ for $j\ge 1$ are considerably more involved, because the restriction on the length of twists for Proposition \ref{Prop:Twistmoll} means we cannot take long enough Dirichlet polynomials to approximate the indicator functions for all possible values of $\mathscr{Y}_j.$ Clearly, the distribution of $\cos\left(\theta_p\right)$ is symmetric about $0$, meaning that, without the restriction of the indicator function, a jump of approximately $\alpha\left(n_j-n_{j-1}\right)$ for the walk  to go from the interval $\left[L_{j-1},U_{j-1}\right]$ at $n_{j-1}$ to $\left[L_j,U_j\right]$ at $n_j$ is about as likely as a fall of $-\alpha\left(n_j-n_{j-1}\right).$
Thus the indicator function must be approximated significantly beyond this range, which supports the choice of $X_j$ in Equation \eqref{chooseXeqn}. Since the random variable $\mathscr{Y}_j$ may not lie in the range $X_j$ of the target interval, where $\mathscr{D}_j^-$ resembles an indicator function, we must use Markov's inequality to show these large values are unlikely. This leads to the following proposition.
\begin{prop}\label{Prop:SteinhausprobsJ}
			Let $(\mathscr{N}_j, 1\leq j\leq \mathscr J)$ be independent Gaussian random variables with mean $0$ and variance $\sigma_j^2$ as given in \eqref{eqn:Sigmaj}.
We have
\begin{equation}\label{jfar} \E\left[\mathscr{Y}_j^{2k_{j,\mathbf{u}}} \left(\mathscr{D}_j^-(\mathscr{Y}_j-u_j)^2+\epsilon_j\right)\frac{(\mathscr{Y}_j-u_j)^{2X_j}}{X_j^{2X_j}}\right]\ll \frac{\E\left[\mathscr{N}_j^{2k_{j,\mathbf{u}}}\mathbf{1}(\mathscr{N}_j\in [u_j , u_j+\Delta_j^{-1}]) \right]}{2^{X_{j}}}.
\end{equation}
\begin{equation}\label{jclose}
		\E\left[\mathscr{Y}_j^{2k_{j,\mathbf{u}}} \left(\mathscr{D}_j^-(\mathscr{Y}_j-u_j)^2+\epsilon_j\right)\right]=\E\left[\mathscr{N}_{j}^{2k_{j,\mathbf{u}}}\mathbf{1}(\mathscr{N}_j \in [u_j, u_j+\Delta_j^{-1}])\right]\left(1+O\left(\Delta_j^{-\frac{1}{4}}\right)\right).
\end{equation}
	\end{prop}
 
\begin{proof}
We start with the inequality \eqref{jfar}, which is the most involved. The expectation is dominated by the region where $\left|\mathscr{Y}_j-u_j\right|>X_j.$ Here, there is little information about $\mathscr{D}_j^-$; it ceases to approximate the indicator function for an interval for these large values. In Section 7.3 of \cite{ABR20} these large values are handled via the Cauchy-Schwarz inequality. However, this would disrupt the tilting and ultimately prevent us from achieving a good constant in Theorem \ref{LBLD}. Instead, we rely on the bounds on the coefficients and use Lemma \ref{Lem: Dbound} for such large inputs into the polynomials $\mathscr{D}^{\pm}.$

The left-hand side of \eqref{jfar} is
\begin{align}
\label{eqn:LargeSteinhaus}=&\E\left[\mathscr{Y}_j^{2k_{j,\mathbf{u}}} \left(\mathscr{D}_j^-(\mathscr{Y}_j-u_j)^2+\epsilon_j\right)\frac{(\mathscr{Y}_j-u_j)^{2X_j}}{X_j^{2X_j}}\mathbf{1}(|\mathscr{Y}_j-u_j|< X_j)\right]+\\
&\nonumber\E\left[\mathscr{Y}_j^{2k_{j,\mathbf{u}}} \left(\mathscr{D}_j^-(\mathscr{Y}_j-u_j)^2+\epsilon_j\right)\frac{(\mathscr{Y}_j-u_j)^{2X_j}}{X_j^{2X_j}}\mathbf{1}(|\mathscr{Y}_j-u_j|\ge X_j)\right].
\end{align}
The first expectation is negligible, and this can be verified when we turn to prove Equation \eqref{jclose}. For the second expectation, we can use Lemma \ref{Lem: Dbound} to bound it as
\begin{equation}
\ll \E\left[\mathscr{Y}_j^{2k_{j,\mathbf{u}}} \left(\frac{2(\mathscr{Y}_j-u_j)}{X_j}\right)^{200X_j\Delta_j^{3a}}\frac{(\mathscr{Y}_j-u_j)^{2X_j}}{X_j^{2X_j}}\right].
\end{equation}
By adding a non-negative quantity, we may bound this as:

\begin{equation}
\le \frac{1}{2^{2X_j}}\E\left[\mathscr{Y}_j^{2k_{j,\mathbf{u}}} \left(\frac{2(\mathscr{Y}_j-u_j)}{X_j}\right)^{2r_j}+\mathscr{Y}_j^{2k_{j,\mathbf{u}}}\left(\frac{2(\mathscr{Y}_j+u_j)}{X_j}\right)^{2r_j}\right],
\end{equation}
where $r_j= X_j+100X_j\Delta_j^{3a}$. This expansion now has the benefit of just involving even moments of $\mathscr{Y}_j$ with positive coefficients. Using Lemma \ref{Cummom}, as well as the symmetry in the distribution of $\mathscr{N}_j$, we see this is at most:
\begin{equation}
\frac{1}{2^{2X_j}}\E\left[\mathscr{N}_j^{2k_{j,\mathbf{u}}} \left(\frac{2(\mathscr{N}_j-u_j)}{X_j}\right)^{2r_j}+\mathscr{N}_j^{2k_{j,\mathbf{u}}}\left(\frac{2(\mathscr{N}_j+u_j)}{X_j}\right)^{2r_j}\right]=\frac{1}{2^{2X_j-1}}\E\left[\mathscr{N}_j^{2k_{j,\mathbf{u}}}\left(\frac{2(\mathscr{N}_j+u_j)}{X_j}\right)^{2r_j}\right].
\end{equation}
 We recall the following Gaussian moments:
\begin{equation}
\E\left[\mathscr{N}_j^{2k_{j,\mathbf{u}}+2r_j}\right]=\sigma_j^{2(k_{j,\mathbf{u}}+r_j)} \frac{(2k_{j,\mathbf{u}}+2r_j)!}{2^{k_{j,\mathbf{u}}+2r_j} (k_{j,\mathbf{u}}+r_j)!},\qquad \E[\mathscr{N}_j^{2k_{j,\mathbf{u}}}]=\sigma_j^{2k_{j,\mathbf{u}}} \frac{(2k_{j,\mathbf{u}})!}{2^{k_{j,\mathbf{u}}} k_{j,\mathbf{u}}!},
\end{equation}
with the variance $\sigma_j^2$ given in Equation \eqref{eqn:Sigmaj}.
Minkowski's inequality for the shifted Gaussian moments now yields: \begin{align}
\label{eqn: 342}
&\frac{1}{2^{2X_j-1}}\E\left[\mathscr{N}_j^{2k_{j,\mathbf{u}}}\left(\frac{2(\mathscr{N}_j+u_j)}{X_j}\right)^{2r_j}\right]\le \sigma_j^{2k_{j,\mathbf{u}}}\frac{(2k_{j,\mathbf{u}})!}{2^{k_{j,\mathbf{u}}+2X_j-1}k_{j,\mathbf{u}}!} \times\Big(\frac{2\left(\sqrt{2 \sigma_j^2(k_{j,\mathbf{u}}+r_j)}+u_j\right)}{ X_j}\Big)^{2r_j}.
\end{align}
We must compare this to the upper bound in Equation \eqref{jfar}, which is \begin{equation}\label{integrandnormal}
		\frac{\int^{u_j+\Delta_j^{-1}}_{u_j}z^{2k_{j,\mathbf{u}}}e^{-\frac{z^2}{2\sigma_j^2} } \rd z}{2^{X_{j}}\sqrt{2\pi}\sigma_j }.
\end{equation} 
Since $z\in [u_j,u_j+\Delta_j^{-1}]$, the integrand is \begin{align}
	u_j^{2k_{j,\mathbf{u}}}e^{-\frac{u_j^2}{2\sigma_j^2}}&\exp\left(2k_{j,\mathbf{u}}\log(1-\frac{u_j-z}{u_j})+\frac{u_j^2-z^2}{2\sigma_j^2}\right)=\\\nonumber
	u_j^{2k_{j,\mathbf{u}}}e^{-\frac{u_j^2}{2\sigma_j^2}}&\exp\left((u_j-z)(\frac{u_j+z}{2\sigma_j^2}-\frac{k_{j,\mathbf{u}}}{u_j})+O\left(\frac{k_{j,\mathbf{u}}\Delta_j^{-2}}{u_j^2}\right)\right)\\&\nonumber\ge \frac{1}{2} u_j^{2k_{j,\mathbf{u}}}e^{-\frac{u_j^2}{2\sigma_j^2}}.
\end{align} 
For the last bound, we used that $\frac{u_j+z}{2\sigma_j^2}=\kappa^*+O(1)$ by Equation \eqref{eqn:barrier} and the fact that  $\frac{k_{j,\mathbf{u}}}{u_j}$ is also $\kappa^*+O(1)$ by Equation \eqref{keqn}. Hence, the upper bound in Equation \eqref{jfar} is greater than: \begin{equation}
	\frac{u_j^{2k_{j,\mathbf{u}}}e^{-\frac{u_j^2}{2\sigma_j^2}}}{2^{X_{j}+1}\Delta_j \sqrt{2\pi}\sigma_j}. 
\end{equation}
Using Stirling's approximation, we have \begin{align}
	\sigma_j^{2k_{j,\mathbf{u}}}\frac{(2k_{j,\mathbf{u}})!}{2^{k_{j,\mathbf{u}}}k_{j,\mathbf{u}}!}&\ll
	 \left(\frac{\sigma_j^2 k_{j,\mathbf{u}}}{e}\right)^{k_{j,\mathbf{u}}}\\
	&\nonumber\ll e^{-\frac{u_j^2}{2\sigma_j^2}} u_j^{2k_{j,\mathbf{u}}} e^{\frac{u_j^2}{2\sigma_j^2}-k_{j,\mathbf{u}}} \left(\frac{2\sigma_j^2 k_{j,\mathbf{u}}}{u_j^2}\right)^{k_{j,\mathbf{u}}}\\&\nonumber\ll e^{-\frac{u_j^2}{2\sigma_j^2}} u_j^{k_{j,\mathbf{u}}} e^{\frac{U_j-L_{j-1}}{2\sigma_j^2}(U_j-L_{j-1}-2\kappa^*\sigma_j^2)}\left(\frac{u_j(L_j-U_{j-1}) }{2\sigma_j^2k_{j,\mathbf{u}} }\right)^{-k_{j,\mathbf{u}}}.  
\end{align}
Since by Equation \eqref{chooseXeqn}, $X_j\gg \frac{k_{j,\mathbf{u}}}{\alpha},$ we see  \begin{equation}
k_{j,\mathbf{u}}<\frac{r_j\kappa^*}{\exp(10^4)},\quad U_{j}-L_{j-1}<\frac{r_j}{\exp(10^4)}	,
\end{equation} and so the contribution in Equation \eqref{eqn: 342} is

\begin{align}
	\ll& \frac{\Delta_j\sigma_j}{2^{X_j}}\E\left[\mathscr{N}_{j}^{2k_{j,\mathbf{u}}}\mathbf{1}(\mathscr{N}_j \in [u_j, u_j+\Delta_j^{-1}])\right] \\&\nonumber\left( \left(\frac{u_j(L_j-U_{j-1}) }{2\sigma_j^2k_{j,\mathbf{u}} }\right)^{-\frac{\kappa^*}{\exp(10^4)}} e^{\frac{U_j-L_{j-1}-2\kappa^*\sigma_j^2}{\exp(10^4)\sigma_j^2}}\Big(\frac{\sqrt{2 \sigma_j^2(k_{j,\mathbf{u}}+r_j)}+u_j}{ X_j}\Big)^{2}
	\right)^{r_j}.
\end{align}
From Equations \eqref{eqn:barrier} and \eqref{keqn}, we see \begin{equation}
	\left(\frac{u_j(L_j-U_{j-1}) }{2\sigma_j^2k_{j,\mathbf{u}} }\right)^{-\frac{\kappa^*}{\exp(10^4)}}\le  \left(1+\frac{10^5\log_{j+1}T}{\sigma_j^2\kappa^*}\right)^{\frac{\kappa^*}{\exp(10^4)}}\le 1.1, 
\end{equation}
whilst $e^{\frac{U_j-L_{j-1}-2\kappa^*\sigma_j^2}{\exp(10^4)\sigma_j^2}}\le 2.$

Finally, from the choice of $X_j$ in Equation \eqref{chooseXeqn}, we see that, for $T$ sufficiently large, \begin{equation}\label{bound1}
X_j>20\left(\sqrt{2\sigma_j^2(k_{j,\mathbf{u}}+r_j)}+(U_j-L_{j-1}+1)\right),
\end{equation}
so that the last term in \eqref{eqn: 342} is $\le 2^{-r_j}$. Thus,
\begin{align}
\nonumber \E\left[\mathscr{N}_j^{2k_{j,\mathbf{u}}}\left(\frac{2(\mathscr{N}_j+u_j)}{X_j}\right)^{2r_j}\right]&\le\E\left[\mathscr{N}_{j}^{2k_{j,\mathbf{u}}}\mathbf{1}(\mathscr{N}_j \in [u_j, u_j+\Delta_j^{-1}])\right] 2^{-X_j-r_j},\end{align}
for $T$ sufficiently large, where we used that $\sqrt{n_j-n_{j-1}}\le 2^{n_j-n_{j-1}-\frac{1}{2}}.$ This completes the proof of  \eqref{jfar}.

We now prove Equation \eqref{jclose}.
The expectation can be split into two cases with $\mathbf{1}(|\mathscr{Y}_j-u_j|\le X_j)$ and $\mathbf{1}(|\mathscr{Y}_j-u_j|> X_j)$.
For the lower bound for $\mathscr{D}^-$, we may clearly discard the second expectation as it is non-negative.
We bound it nevertheless, since it would be needed for an upper bound when dealing with $\mathscr{D}^+$, which is necessary for the fourth moment in the Paley-Zygmund inequality.

The first expectation $\E\left[\mathscr{Y}_j^{2k_{j,\mathbf{u}}} \left(\mathscr{D}_j^-(\mathscr{Y}_j-u_j)^2+\epsilon_j\right)\mathbf{1}(|\mathscr{Y}_j-u_j|\le X_j)\right]$ is handled as in the proof of Proposition \ref{Prop:Steinhausprobsinit} to express it as: \begin{equation}
\E\left[\mathscr{Y}_{j}^{2k_{j,\mathbf{u}}}\mathbf{1}(\mathscr{Y}_j \in [u_j, u_j+\Delta_j^{-1}])\right]\exp\left(O\left(\Delta_j^{-\frac{1}{4}}\right)\right).
\end{equation}
The tight bounds on $\mathscr{Y}_j^{2k_{j,\mathbf{u}}}$, when $\mathscr{Y}_j$ is restricted to the interval $[u_j, u_j+\Delta_j^{-1}]$, are sufficient for a Gaussian comparison as in \cite[Lemma 8]{AB24} for $j\ge 1$, to rewrite this as
\begin{equation}\E\left[\mathscr{N}_{j}^{2k_{j,\mathbf{u}}}\mathbf{1}(\mathscr{N}_j \in [u_j, u_j+\Delta_j^{-1}])\right]\left(1+O\left(\Delta_j^{-\frac{1}{4}}+\epsilon_j\Delta_j\log_{j+2} T\right)\right).\end{equation}

Finally, we may bound the expectation with $\mathbf{1}(|\mathscr{Y}_j-u_j|> X_j)$:
\begin{equation}
\E\left[\mathscr{Y}_j^{2k_{j,\mathbf{u}}} \left(\mathscr{D}_j^-(\mathscr{Y}_j-u_j)^2+\epsilon_j\right)\mathbf{1}(|\mathscr{Y}_j-u_j|> X_j)\right] \le \E\left[\mathscr{Y}_j^{2k_{j,\mathbf{u}}} \left(\mathscr{D}_j^-(\mathscr{Y}_j-u_j)^2+\epsilon_j\right)\frac{\left(\mathscr{Y}_j-u_j\right)^{2X_j}}{X_j^{2X_j}}\right] ,
\end{equation}
for which we apply \eqref{jfar}.  

We can now bound all expectations in \eqref{eqn: steinhaus expect}. Using Proposition \ref{Prop:SteinhausprobsJ}, the main term is
\begin{align}
\E&\left[ \left|\zeta M\right|^2 Y_{0}^{2k_{0,\mathbf{u}}}\left(\mathscr{D}_0^-(Y_0-u_0)^2+\epsilon_0\right)\prod_{1\le j\le \mathscr{J}}Y_{j}^{2k_{j,\mathbf{u}}} \left(\mathscr{D}^-_j(Y_j-u_j)^2+\epsilon_j\right)\mathbf{1}\left(|Y_j-u_j|\le X_j\right)\right]\nonumber \\&\gg \frac{\log T}{\log T_{\mathscr{J}}}\E\left[\mathscr{Y}_{0}^{2k_{0,\mathbf{u}}}\mathbf{1}(\left|\mathscr{M}_0\right|^2 \in [u_0, u_0+\Delta_0^{-1}])\right]\E\left[\prod^{\mathscr{J}}_{j=1}\mathscr{N}_{j}^{2k_{j,\mathbf{u}}}\mathbf{1}(\mathscr{N}_j \in [u_j, u_j+\Delta_j^{-1}]) \right].
\end{align}

Using Lemma \ref{Lem:MAGIC}, we see that

\begin{align}
	\frac{1}{\mathfrak{c}_{\mathbf{u}}} \E&\left[ \left|\zeta M\right|^2 Y_0^{2k_{0,\mathbf{u}}}\left(\mathscr{D}_0^-(Y_0-u_0)^2+\epsilon_0\right)\prod_{1\le j\le \mathscr{J}} Y_j^{2k_{j,\mathbf{u}}}\left(\mathscr{D}^-_j(Y_j-u_j)^2+\epsilon_j\right)\mathbf{1}\left(|Y_j-u_j|\le X_j\right)\right]\nonumber \\&\gg_\alpha \frac{\log T}{\log T_{\mathscr{J}}}\E\left[\left(\frac{\mathscr{Y}_{0}}{z_0}\right)^{2k_{0,\mathbf{u}}}\mathbf{1}(\left|\mathscr{M}_0\right|^2 \in [u_0, u_0+\Delta_0^{-1}])\right]\E\left[\prod^{\mathscr{J}}_{j=1}\mathbf{1}(\mathscr{N}_j \in [u_j, u_j+\Delta_j^{-1}]) \right],
	\end{align}
	where for convenience we set $z_0=-\frac{1}{2}\log { u_0}.$

This expectation is easiest to handle by performing the calculations over the Gaussian variables $\mathscr{N}_j$ with $j\ge 1$ first. For a fixed $u^*_0\in [L_0,U_0],$ we sum over all combinations $\mathbf{u}=(u_0,...,u_{\mathscr{J}})$ in $ \mathfrak{I}^-$, to attain:
\begin{align}
\sum_{\mathbf{u}\in \mathfrak{I}^-;u_0=u^*_0}\frac{1}{\mathfrak{c}_{u}} \E&\left[ \left|\zeta M\right|^2 Y_0^{2k_{0,\mathbf{u}}}\left(\mathscr{D}_0^-(Y_0-u_0^*)^2+\epsilon_0\right)\prod_{1\le j\le \mathscr{J}}Y_j^{2k_{j,\mathbf{u}}} \left(\mathscr{D}^-_j(Y_j-u_j)^2+\epsilon_j\right)\mathbf{1}\left(|Y_j-u_j|\le X_j\right)\right]\nonumber\\&\gg \E\left[\left(\frac{\mathscr{Y}_{0}}{z^*_0}\right)^{2k_{0,\mathbf{u}}}\mathbf{1}(\left|\mathscr{M}_0\right|^2 \in [u_0^*, u_0^*+\Delta_0^{-1}])\right]\E\left[\prod^{\mathscr{J}}_{j=1}\mathbf{1}\left(z^*_0+\sum^j_{l=1}\mathscr{N}_l \in [L_j,U_j]\right) \right]\label{eqn:firstcalc},
\end{align}
where $z_0^*=-\frac{1}{2}\log u_0$ and we have a slight abuse of notation justified since $k_{0,\mathbf{u}}=\left\lceil \kappa^*z_0^*\right\rceil$ is independent of $u_1,...,u_{\mathscr{J}}.$

It is an easy exercise using the tilt and barrier to perform the expectation over the normal variables in the second expectation:
\begin{equation} \label{eqn:Firstgauss}
\asymp_\alpha  \P\Big(z^*_0+\sum^{\mathscr{J}}_{l=1}\mathscr{N}_l \in [L_{\mathscr{J}},U_{\mathscr{J}}]\Big),
\end{equation}

Since $\sum^{\mathscr{J}}_{l=1}\mathscr{N}_l$ is a centred Gaussian with variance $\sigma^2=\sum_{p\in \cup^{\mathscr{J}}_{j=1}\mathscr{P}_j} \frac{1}{2p}+\frac{1}{4p^2}$, the above is
\begin{equation}
\label{eqn:Polytilt}
\asymp_\alpha  \int^{\infty}_{L_{\mathscr{J}}}\frac{e^{-\frac{(z-z^*_0)^2}{2\sigma^2}}}{\sigma}\rd z\nonumber  
\asymp_\alpha  \frac{e^{\frac{-\left(L_{\mathscr{J}}-z^*_0\right)^2}{2\sigma^2} }}{\sigma}.
\end{equation}

We must now reconcile the two random variables $\mathscr{Y}_0$ and $\mathscr{M}_0$ in \eqref{eqn:firstcalc} to evaluate the expectation.
Firstly, we need to convert the polynomial tilt $\mathscr{Y}_0^{2k_{0,\mathbf{u}}}$ into an exponential tilt.
Observe that: \begin{equation}
-\log \left|\mathscr{M}_0\right|=\label{eqn:mlog} \sum_{k\ge 1}\sum_{p\in \mathscr{P}_0} \frac{\cos k\theta_p}{p^{\frac{k}{2}}}\quad \in[\mathscr{Y}_0-10, \mathscr{Y}_0+10],
\end{equation}
by bounding the terms with $k\ge 3$ in the sum pointwise. This range is sufficiently small that we may convert the tilt into an exponential by Taylor expansion:
\begin{equation}\label{eqn:exptilt}
\left(\frac{\mathscr{Y}_0}{z^*_0}\right)^{2k_{0,\mathbf{u}}} \asymp_\alpha  \exp\left(2k_{0,\mathbf{u}}\left(\frac{ \mathscr{Y}_0}{z^*_0}-1\right)\right) .
\end{equation}
Combining this with Equations \eqref{eqn:Firstgauss} and  \eqref{eqn:Polytilt} gives that the right-hand side of the bound in Equation \eqref{eqn:firstcalc} is
\begin{equation}
 \frac{e^{\frac{-\left(L_{\mathscr{J}}-z^*_0\right)^2}{2\sigma^2} }}{\sigma} 
\E\Big[\exp\Big(2k_{0,\mathbf{u}}\left(\frac{\mathscr{Y}_0}{z^*_0}-1\right)\Big)\mathbf{1}\left(\left|\mathscr{M}_0\right|^2 \in \left[e^{-2z^*_0}, e^{-2z^*_0}+\Delta_0^{-1}\right]\right)\Big].
\end{equation}
Using this, Equation \eqref{eqn:mlog} and the definition of $L_{\mathscr{J}}$, as well as the restrictions of $|\mathscr{M}_0|^2$ given by the indicator function, the above simplifies to
\begin{align}\label{eqn:MidGauss}
\asymp_\alpha& \exp\left(-2k_{0,\mathbf{u}}\right)\frac{e^{-\frac{\left(V^{*}-z^*_0\right)^2}{n_{\mathscr{J}}-n_0}}}{\sqrt{n_{\mathscr{J}}-n_0}}\cdot\E\big[\exp\left(-2\kappa^\ast\log\left|\mathscr{M}_0\right|\right)\mathbf{1}(\left|\mathscr{M}_0\right|^2 \in [e^{-2z^*_0}, e^{-2z^*_0}+\Delta_0^{-1}])\big],
\end{align}

To evaluate the expectation, note that the random variables $(\theta_p, {p\text{ prime}})$ remain independent in the tilted measure: $$\widetilde{\E}\left[\bullet\right]=\frac{\E\left[\bullet \exp\left(-2\kappa^\ast\log\left|\mathscr{M}_0\right|\right)\right]}{ \E\left[\exp\left(-2\kappa^\ast\log\left|\mathscr{M}_0\right|\right)\right]},$$
with the associated tilted probability of an event $A$ defined as $\widetilde{\P}(A)=\widetilde{\E}\left[\mathbf{1}(A)\right].$
It is then more convenient to work with the variable
\begin{equation}\label{eqn:R}
R_0=\log\left|\mathscr{M}_0\right|+\kappa^\ast n_0.
\end{equation}
The expectation becomes
\begin{equation}\label{Ceqn}
e^{2\kappa^{\ast2}n_0}\E\left[\exp\left(-2\kappa^\ast R_0\right)\right]\cdot \widetilde{\P}\left(R_0\in\left[\kappa ^\ast n_0-z^*_0,\kappa^* n_0-z^*_0+\delta_0\right]\right),
\end{equation}
where the increments $$\delta_0=\delta_0(z^*_0,\Delta_0)=\log\left(1+\Delta_0^{-1}e^{2z^*_0}\right)$$ vary mildly due to the mapping in Equation \eqref{eqn:R}.
Using Lemma 12 from \cite{AB24}, we have \begin{equation}
\E\left[\exp\left(-2\kappa^\ast R_0\right)\right]\asymp_\alpha a_\alpha e^{\gamma\alpha^2} e^{- \kappa^{\ast2}n_0}.
\end{equation}
Since we have tilted to favour the range, it is easy to approximate the $\widetilde{\P}$-probability in \eqref{Ceqn} at each $z^*_0$ following \cite[Lemma 18]{ABR20}, which enables us to approximate the sum of the independent variables at each prime with a single random variable with a Gaussian distribution.
Under the mapping \eqref{eqn:R}, the condition $u^*_0\in [L_0,U_0]$ becomes $R_0\in \left[0, n_0^{\frac{2}{3}}/2\right].$
Using this in Equation \eqref{eqn:MidGauss} and performing the calculation over the Gaussian convolution from the intervals $p\le T_0$ and $p\in (T_0,T_{\mathscr{J}}]$ we attain: 
	\begin{equation}
\E\left[\left|\zeta M\right|^2 \mathbf{1}\left(G_{\mathscr{J}}\right)\right]\asymp_\alpha a_\alpha e^{\gamma\alpha^2} \frac{e^{-\frac{V^{*2}}{\log\log T_{\mathscr{J}}}}}{\sqrt{\log\log T_{\mathscr{J}}}}.\end{equation}
Equation  \eqref{2twisteqn} follows by substituting the definition of $V^*$ and $T_{\mathscr J}$.

The proof of Equation \eqref{notwisteqn} and \eqref{4twisteqn} are similar. For  \eqref{4twisteqn} we  calculate the fourth twisted moment as in the proof of Equation \eqref{2twisteqn}, proceeding similarly to Equation \eqref{eqn:boundindic}, but using Equation \eqref{4} for the fourth twisted mollified moment to pass to the Steinhaus model, in place of Equation \eqref{2}.
Meanwhile, for Equation \eqref{notwisteqn}, we use the Mean Value Theorem for Dirichlet polynomials in place of Equation \eqref{2} in the above proof  now yields Equation \eqref{notwisteqn}, which completes the proof of Proposition \ref{prop:goodprobs}.\end{proof}
\begin{rem}
	\rm
	We believe that the truncation $T^\ast=T^{\frac{1}{\mathscr{R}(\alpha^2+1)}}$ is the best possible with our techniques, and so the constant $K_\alpha$ cannot be improved beyond the shape in Equation \eqref{ca}. To see this, write $P_j(S_j-S_{j-1})$ for any possible Dirichlet polynomial representing a tilt and the restriction in the interval $\mathscr{P}_j$.  In our proof, we have taken the tilt  to depend on the interval $\mathbf{u}\in\mathfrak{I}^{\pm}.$ Here, $P_j(S_j-S_{j-1})=(S_j-S_{j-1})^{2k_{j,\mathbf{u}}}D^{\pm}_j((S_j-S_{j-1})-u_j)^2,$ which favours jumps consistent with the partial sum over the interval $\mathscr{P}_j$ lying in the relevant boundaries.
		We now argue that we must have
		\begin{equation}
			\label{eqn: deg Pj}
			\deg P_j\gg \alpha^2(n_j-n_{j-1}).
		\end{equation}
		We now consider the contribution to the expectation $\E\left[ \left|P_j(S_j-S_{j-1})\right|\right]$ from different ranges of values for $S_j-S_{j-1}$. To be a good approximation to the probability we remain within the barrier, a large contribution must come from when $S_j-S_{j-1}\approxeq u_j$. However, this is less likely than smaller values of $S_j-S_{j-1}$ due to Gaussian decay, and so $P_j$ must have a high degree to make the interval dominate the expectation.

		More concretely, for $x\approx u_j$  in the target increment for the interval $\mathscr{P}_j$ and $y\approx \frac{\alpha(n_j-n_{j-1})}{2}$ (say), we should have
		\begin{equation}\label{eqn:GenTilt}\frac{e^{-\frac{x^2}{n_{j}-n_{j-1}}}}{\sqrt{n_{j}-n_{j-1}}}\int^{x+\Delta_j^{-1}}_x P_j(z)\rd z\gg \frac{e^{-\frac{y^2}{n_{j}-n_{j-1}}}}{\sqrt{n_{j}-n_{j-1}}}\int^{y+\Delta_j^{-1}}_y P_j(z)\rd z.\end{equation}
		The ratio of the Gaussian factors is approximately $e^{-3\alpha^2(n_j-n_{j-1})/4}.$
		We use the basis of shifted Chebyshev polynomials $\left\{T_n\left(\frac{  z'-\alpha(n_j-n_{j-1})/2}{\alpha(n_j-n_{j-1})/100}\right)\right\}$  for the polynomial $\mathfrak{F}_j(z')=\int^{z'+\Delta_j^{-1}}_{z'} P_j(z)\rd z$ and the (shifted) orthogonality relations of these Chebyshev polynomials to relate the coefficients in this basis to $$\sup_{y\approx \frac{\alpha(n_j-n_{j-1})}{2}} \mathfrak{F}_j(y).$$
		Using a pointwise bound for $\mathfrak{F}_j(x)$ when $x$ is close to the target jump, one can thus show that the minimum degree  of any such polynomial used in the twist must obey \eqref{eqn: deg Pj}.
	
	\end{rem}

\appendix
\section{Second and fourth twisted mollifier formulae}\label{Sec:Twistmoll}

\begin{proof}[Proof of Proposition \ref{Prop:Twistmoll}]
We start by proving Equation \eqref{2}.
In \cite{HB85}, twisted moments of the zeta function by the square of the absolute value of a Dirichlet polynomial are considered:

\begin{equation}
I=\int^T_0 \left|\zeta\left(\frac{1}{2}+it\right)\right|^2 \left|A\left(\frac{1}{2}+it\right)\right|^2 \rd t, \qquad A(s)=\sum_{m\le M} a(m)m^{-s},
\end{equation}
with $a(m)\ll_\epsilon  m^\epsilon.$
They prove the following theorem.
\begin{thm}[Theorem 1 in \cite{HB85}]
\begin{equation}
I=T\sum_{h,k\le M} \frac{a(h)\overline{a(k)}}{[h,k]}\left(\log \frac{T(h,k)}{2\pi [h,k]}+2\gamma-1\right)+\mathscr{E},
\end{equation}
with $\mathscr{E}=\mathscr{E}_1+\mathscr{E}_2, $ with $\mathscr{E}_1\ll_B T(\log T)^{-B}$ for any $B>0$ and $\mathscr{E}_2\ll_\epsilon M^2T^\epsilon$ for any $\epsilon>0.$
\end{thm}
The same proof shows that
if we define the coefficients of the mollifier such that \begin{equation}
M(t)=\sum_n \frac{u(n)}{n^s},
\end{equation}
and $Q$ is $\alpha$-separable
then $\int^{T}_0\left|\zeta M \right|^2Q(t)^2 \rd t$ is
\begin{align}
T\sum_{c,g} b(c,g)\sum_{\substack{m=cf,n=gh\\p|fh\Rightarrow p\in \mathscr{P}_j}}   \frac{u(f)u(h)}{[m,n]} \left(\log \frac{T[m,n]}{2\pi (m ,n)} +2\gamma-1\right)&+O(T\log^{-100(\alpha^2+1)}Tb(1,1)).
\end{align}

The calculations to prove Equation \eqref{2} follow that of the proof of Theorem 1.5 in \cite{AC25}: using Rankin's trick to lift the restriction on the number of prime factors in each interval allowed in the support of the coefficients comprising the mollifier and handling the logarithm in the same way, and then evaluating the unrestricted sum.
We split the logarithm as in the proof of Theorem 1.5 in \cite{AC25} to express this as \begin{equation}\label{eqn: splitlog}
TI_1(\mathbf{b})+Tb(1,1)O\left(I_2(\mathbf{b})\right) +O(T\log^{-100(\alpha^2+1)}Tb(1,1)),
\end{equation} where \begin{equation}\label{I1eqn}
I_1(\mathbf{b})=\sum_{c,g} \frac{b(c,g)}{[c,g]} \left(\log\left(\frac{T[c,g]}{2\pi(c,g)}\right)+2\gamma-1\right)\sum_{\substack{m=cf,n=gh\\p|fh\Rightarrow p\in \mathscr{P}_j}}   \frac{u(f)u(h)[c,g]}{[m,n]}
\end{equation}
and
\begin{equation}
I_2(\mathbf{b})=\sum_{\substack{p|cg\Rightarrow p\le T_{\mathscr{J}}\\\Omega_0(cg)\le \log^{\frac{1}{100}}T\\ \Omega_j(cg)\le  \alpha^2(n_j-n_{j-1})^{10^4}\forall 1\le j\le \mathscr{J}\\ (c,g)=1}} \frac{1}{[c,g]}\left|\sum_{\substack{m=cf,n=gh\\p|fh\Rightarrow p\in \mathscr{P}_j}}   \frac{u(f)u(h)[c,g]}{[m,n]} \log\left( \frac{[m,n](c,g)}{(m ,n)[c,g]}\right)\right|.
\end{equation}
We bound these sums in the following Lemma:
\begin{lem}\label{lem: Evaluate} The quantities $I_1(\mathbf{b})$ and $I_2(\mathbf{b})$ as defined above obey the following bounds: \begin{equation}\label{I1evaleqn}
I_1(\mathbf{b})\asymp \frac{\log T}{\log T_{\mathscr{J}}}b(1,1) ,\qquad I_2(\mathbf{b})\ll 1.
\end{equation}
\end{lem}
Combining these bounds and substituting them into Equation \eqref{eqn: splitlog} completes the proof of Equation \eqref{2}.
﻿
We begin evaluating $I_1(\mathbf{b})$ by handling the inner sum \begin{equation}
\sum_{\substack{m=cf,n=gh\\p|fh\Rightarrow p\in \mathscr{P}_j}}   \frac{u(f)u(h)[c,g]}{[m,n]}
\end{equation}
from Equation \eqref{I1eqn}.
﻿
This sum splits as:
\begin{equation}
\label{eqn:mainsplit}\prod_{0\le j \le \mathscr{J}} C^{(1)}_j(c_j,g_j), \quad \text {where } c=\prod_{0\le j \le \mathscr{J}}c_j,\quad g=\prod_{0\le j \le \mathscr{J}}g_j
\end{equation}
with $p|c_jg_j\Rightarrow p\in \mathscr{P}_j,$ and
\begin{equation}
C^{(1)}_j(w,z)=\sum_{\substack{m=wf,n=zh\\ p|fh\Rightarrow p\in \mathscr{P}_j}}   \frac{u(f)u(h)[w,z]}{[m,n]},
\end{equation}
whenever $(w,z)=1, p|wz\Rightarrow p\in \mathscr{P}_j$ and $\Omega_j(wz)\le 2\alpha^2(n_j-n_{j-1})^{10^4}\forall 1\le j\le \mathscr{J}$.
There is no truncation involved for the mollifier in the interval $\mathscr{P}_0$ and so \begin{align}
C_0^{(1)}(w,z)=\sum_{\substack{m=wf,n=zh\\ p|fh\Rightarrow p\in \mathscr{P}_0}}   \frac{\mu(f)\mu(h)[w,z]}{[m,n]}.
\end{align}

Motivated by the generalised Von Mangoldt function (e.g. see \cite{GPY09}), we see that for any set $\mathscr{P}$ of primes and $r\in \Z_{\ge 0}$, if the prime factors of some integers $u$ and $v$ are contained in $\mathscr{P}$, then
\begin{equation}\label{Seleqn}
\sum_{\substack{m=uf\\ n=vh\\p|fh\Rightarrow p\in \mathscr{P} }} \frac{\mu(f)\mu(h)[u,v]}{[m,n]}\log^{r}\left(\frac{[u,v]}{[m,n]}\right)=0,
\end{equation}
whenever the prime factorisations of $u$ and $v$ differ by more than $r$ primes with multiplicity.
Taking $\mathscr{P}=\mathscr{P}_0$ and $r=0,$ we see $C_0^{(1)}(w,z)$ vanishes unless $w=z$ (which must be $1$ as $(w,z)=1$), when multiplicativity yields: \begin{align}\label{C0eqn}
C_0^{(1)}(1,1)&=\sum_{p|fh\Rightarrow p\in \mathscr{P}_0} \frac{\mu(f)\mu(h)}{[f,h]}
=\prod_{p\in \mathscr{P}_0} \left(1-\frac{1}{p}\right)\nonumber.
\end{align}
Using Rankin's trick, we may estimate the terms $C_j^{(1)}(w,z)$ for $1\le j\le \mathscr{J}.$
Because the index for the number of prime factors used in the mollifier $M(t),$ $5\alpha^2(n_j-n_{j-1})^{10^5},$ is sufficiently larger than the index for the number of prime factors allowed in the support of $\alpha$-separable twists, $2\alpha^2(n_j-n_{j-1})^{10^4},$ the same method of proof works as in Lemma 3.6 in \cite{AC25}. This yields:  
\begin{align}
 I_1(\mathbf{b})&=\prod_{p\le T_{\mathscr{J}}}(1-\frac{1}{p})\times\\\nonumber\sum_{\substack{c,g\\ c_0=g_0=1}} \frac{b(c,g)}{[c,g]} \left(\log\left(\frac{T[c,g]}{2\pi(c,g)}\right)+2\gamma-1\right)&\prod_{1\le j\le \mathscr{J}}
\left\{\mathbf{1}(c_j=g_j=1)+O\left(e^{-100(n_j-n_{j-1})}\right)\right\}.
\end{align}    

Moreover, using H\"older's inequality on the final product, we attain: \begin{equation}\label{Chern2aeqn}
I_1(\mathbf{b})\asymp b(1,1)\log T\prod_{p\le T_{\mathscr{J}}}\left(1-\frac{1}{p}\right).
\end{equation}
A simple application of the Prime Number Theorem means we may write the right-hand side of Equation \eqref{Chern2aeqn} as $\asymp b(1,1)\frac{\log T}{\log T_{\mathscr{J}}}$, as claimed.

We now  prove the bound $I_2(\mathbf{b})\ll 1$. If we were able to remove the truncation, then Equation \eqref{Seleqn} would yield that the sum is only supported where the coefficients differ by at most one prime factor, as we see in Section 3.2 of \cite{AC25}. Indeed, the use of Rankin's trick and evaluation follows the proof of Proposition 5 in \cite{AC25} very closely, remarking again that the mollifier uses many more prime factors than the twist, and in place of Equation (3.168) we attain:
\begin{align}
|I_2(\mathbf{b})|\ll \prod_{p\le T_{\mathscr{J}}}\left(1-\frac{1}{p}\right)\sum^{\mathscr{J}}_{j=0} e^{n_j}
\ll \prod_{p\le T_{\mathscr{J}}}\left(1-\frac{1}{p}\right) \log T_{\mathscr{J}}\ll 1.
\end{align}

This completes the proof of Equation \eqref{2}. The proof of Equation \eqref{4} follows the proof of Lemma 9 for the fourth  mollified moment twisted by the square of a Dirichlet polynomial in \cite{ABR20}. The changes required to twist by the square of an $\alpha$-separable polynomial instead of a well-factorable Dirichlet polynomial are superficial, and go through as in the proof of Equation \eqref{2}.
\end{proof}
\section{Probabilities in the Steinhaus model}\label{Sec:Steinprob}
We first compare the moments of $A_{p,1}=\cos(\theta_p)p^{-1/2}$ and $A_{p,2}=\cos(2\theta_p)p^{-1}/2$ to those of the Gaussian variable $N_{p,i}=\nu_ip^{-\frac{i}{2}}Z_{p,i}$, where the variables $Z_{p,i}$ are IID standard Gaussian variables for $i=1,2$ with $\nu_1 =\frac{1}{\sqrt{2}}$ and $\nu_2=\frac{1}{2\sqrt{2}}$.  
\begin{lem}\label{YtoN}
For  $i\in \{1,2\}$ and any $k\in \N,$
\begin{equation}\label{oddmom}
\E[A_{p,i}^{2k-1}]=\E[N_{p,i}^{2k-1}]=0,
\end{equation} 
while for the even moments, we have the bound
\begin{equation}\label{evmom}
\E[A_{p,i}^{2k}]\le \E[N_{p,i}^{2k}].
\end{equation}

\end{lem}
\begin{proof}[Proof of Lemma \ref{YtoN}.]
From the odd symmetries of the distributions, Equation \eqref{oddmom} is immediate, while for the even moments in \eqref{evmom}, we may assume $k\ge 2,$ since we have constructed the moments to be equal for $k=1.$ Then we have the $2k$-th moments:
\begin{equation}\label{Amom}
\E\left[\frac{\nu_i^{2k}\cos^{2k}(i \theta_p)}{p^{ik}}\right]=\frac{\nu_i^{2k}{2k\choose k}}{2^{k} p^{ik}}.
\end{equation}

But we recall the Gaussian compute the moments of $N_p$ precisely, to attain:
\begin{align}\label{Nmom}
\E[N_{p,i}^{2k}]&= \frac{\nu_i^{2k}(2k)!}{2^kp^{ik}k!}.
\end{align}
Comparing Equations \eqref{Amom} and \eqref{Nmom} completes the proof of inequality \eqref{evmom}.\end{proof}
These bounds translate into bounds on the sum $\mathscr{Y}_j$ from \eqref{eqn: Yj}.
\begin{lem}\label{Cummom} 
Let $1\le j\le \mathscr{J},$ $i=1,2,$ $\mathscr{S}_{j,i}=\sum_{p\in \mathscr{P_j}} A_{p,i}$ and $\mathfrak{N}_{j,i}=\sum_{p\in\mathscr{P}_j }N_{p,i}$.
Then for any $k\in \N$, we have \begin{equation}\label{oddcummom}
\E\left[\mathscr{S}_{j,i}^{2k-1}\right]=\E[\mathfrak{N}_{j,i}^{2k-1}]=0,
\end{equation}
while for the even moments, we have
\begin{equation}\label{evcummom}
\E\left[\mathscr{S}_{j,i}^{2k}\right]\le\E[\mathfrak{N}_{j,i}^{2k}].
\end{equation}
Hence, if $k\ll (n_j-n_{j-1})T_{j-1}^{\frac{1}{100}},$ then \begin{equation}\label{totalmoment}
\E\left[\mathscr{Y}_j^{2k} \right]\ll \E[\mathscr{N}_j^{2k}]\end{equation} for the centred Gaussian $\mathscr{N}_j$ with variance $\sigma_j^2$ defined in Equation \eqref{eqn:Sigmaj} 
\end{lem}
\begin{proof}[Proof of Lemma \ref{Cummom}]
Again by the odd symmetries of the distributions  we see Equation \eqref{oddcummom} is immediate, while for \eqref{evcummom}, we expand out the power to write for $\mathbf n=(n_p, p\in \mathscr{P}_j)$
\begin{equation}
\E\left[\mathscr{S}_{j,i}^{2k}\right]=\sum_{\substack{n_p\ge 0, \forall p \in \mathscr{P}_j\\ \sum_{p\in \mathscr{P}_j} n_p=2k}}{2k \choose \mathbf{n}}\E\left[\prod_{p\in \mathscr{P}_j} A_{p,i}^{n_p}\right].
\end{equation}
The expectation splits by independence, and we use Equation \eqref{oddmom} to show the expectation vanishes unless all the $n_p$ are even.
Hence we may write the sum as
\begin{equation}
\sum_{\substack{2|n_p, \forall p \in \mathscr{P}_j\\ \sum_{p\in \mathscr{P}_j} n_p=2k}}{2k \choose \mathbf{n}}\prod_{p\in \mathscr{P}_j}\E\left[  A_{p,i}^{n_p}\right].
\end{equation}
Using Equation \eqref{evmom}, we may bound the above sum as:
\begin{equation}
\le \sum_{\substack{2|n_p, \forall p \in \mathscr{P}_j\\ \sum_{p\in \mathscr{P}_j} n_p=2k}}{2k \choose \mathbf{n}}\prod_{p\in \mathscr{P}_j}\E\left[ N_p^{n_p}\right].
\end{equation}
Finally, we use Equation \eqref{oddmom} again to show the other choices with some odd values of $n_p$ don't contribute to the sum, so this is
\begin{equation}
\sum_{\substack{n_p\ge 0, \forall p \in \mathscr{P}_j\\ \sum_{p\in \mathscr{P}_j} n_p=2k}}{2k \choose \mathbf{n}}\prod_{p\in \mathscr{P}_j}\E\left[ N_{p,i}^{n_p}\right]=\E[\mathfrak{N}_{j,i}^{2k}].
\end{equation}
We now turn to bound $\E\left[\left(\mathscr{S}_{j,1}+\mathscr{S}_{j,2}\right)^{2k}\right]$ to prove \eqref{totalmoment}.

By observing that $\mathfrak{N}_{j,i}\sim N(0,\tau_{i,j}^{2})$, where \begin{equation}
	\tau_{i,j}^2=v_i^2\sum_{p\in \mathscr{P}_j}  p^{-i},
\end{equation} we obtain \begin{equation}
	\frac{\E\left[\mathfrak{N}_{j,2}^{2k}\right]}{\E\left[\mathfrak{N}_{j,1}^{2k}\right]}=\left(\frac{\tau_{2,j}}{\tau_{1,j}}\right)^{2k}\ll \left(T_{j-1}\log T_{j-1} (n_j-n_{j-1})\right)^{-k}.
\end{equation}

Since $\mathscr{N}_j\sim N(0,\tau_{1,j}^2+\tau_{2,j}^2)$, Minkowski's inequality and the bounds from \eqref{evmom} now yield

\begin{equation}
	\E\left[\left(\mathscr{S}_{j,1}+\mathscr{S}_{j,2}\right)^{2k}\right]\ll  \left(\frac{(\tau_{1,j}+\tau_{2,j})^2}{\tau_{1,j}^2+\tau_{2,j}^2}\right)^{k} \E[\mathscr{N}_j^{2k}]\sim \E[\mathscr{N}_j^{2k}],
\end{equation} 
with the final asymptotic valid since $k\ll (n_j-n_{j-1})T_{j-1}^{\frac{1}{100}}.$
\end{proof}

\bibliographystyle{alpha}
\bibliography{ref.bib}

\end{document}